\documentclass[onefignum,onetabnum]{siamart171218}

\usepackage{lipsum}
\usepackage{amsfonts}
\usepackage{graphicx}
\usepackage{epstopdf}
\ifpdf
  \DeclareGraphicsExtensions{.eps,.pdf,.png,.jpg}
\else
  \DeclareGraphicsExtensions{.eps}
\fi

\newsiamremark{remark}{Remark}
\newsiamremark{hypothesis}{Hypothesis}
\crefname{hypothesis}{Hypothesis}{Hypotheses}
\newsiamthm{claim}{Claim}

\headers{Least-squares spectral methods}{B. Hashemi and Y. Nakatsukasa }

\title{Least-squares spectral methods for ODE eigenvalue problems\thanks{Version of \today.
\funding{The work of the first author was supported by Iran National Science Foundation
(INSF) grant 98012590.}}}

\author{Behnam Hashemi\thanks{Department of Mathematics, Shiraz University of Technology, Modarres Blvd., Shiraz 71555-313, Iran. (\email{hashemi@sutech.ac.ir}, \url{https://sites.google.com/view/bhashemi}).}
\and 
Yuji Nakatsukasa\thanks{Mathematical Institute, University of Oxford, Oxford, OX2 6GG. (\email{nakatsukasa@maths.ox.ac.uk}, \url{https://people.maths.ox.ac.uk/nakatsukasa/}).}
}

\usepackage{amsopn}
\DeclareMathOperator{\diag}{diag}

\DeclareMathOperator*{\minimize}{minimize}

\newcommand\quasi[1]{\mathsf{#1}} 

\newtheorem{example}{Example}
\usepackage{graphicx}
\graphicspath{{figures/}}

\usepackage[normalem]{ulem}
\usepackage{multirow}
\usepackage{epstopdf}
\usepackage{listings}
\usepackage[]{mcode}
\usepackage{algorithm}
\usepackage{latexsym}
\usepackage{showidx}
\usepackage{latexsym}
\usepackage{amssymb}

\newcommand{\eqnref}[1]{(\ref{#1})}
\newcommand{\ignore}[1]{}

\usepackage{upquote}
\usepackage{graphicx,color}
\usepackage{algcompatible}
\usepackage{latexsym}
\usepackage{amssymb}
\usepackage{amsmath}
\usepackage{pdflscape}

\newcommand\ip[2]{\langle {#1},{#2} \rangle}

\usepackage{mathrsfs}
\usepackage{bigfoot}

\def\Cz{\mathbb{C}}

\begin{document}
\maketitle

\begin{abstract}
We develop spectral methods for ODEs and operator eigenvalue problems
that are based on a least-squares formulation of the problem.\ The key tool is a method for rectangular generalized eigenvalue problems, 
which we extend to quasimatrices and objects combining quasimatrices and matrices. 
The strength of the approach is its flexibility that lies in the quasimatrix formulation allowing the basis functions to be chosen arbitrarily (e.g. those obtained by solving nearby problems), 
and often giving high accuracy. We also show how our algorithm can easily be modified to solve problems with eigenvalue-dependent boundary conditions, and discuss reformulations as an integral equation, which often improves the accuracy. 
\end{abstract}

\begin{keywords}
operator eigenvalue problems, least-squares method, rectangular matrix pencils, spectral methods, quasimatrix
\end{keywords}

\begin{AMS}
47A75, 65N35, 65F25, 65F15
\end{AMS}

\section{Introduction}\label{sec:intro}
Consider a differential eigenvalue problem 
\begin{equation}
  \label{eq:diffeq}
\mathcal{L} u = \lambda u, 
\end{equation}
where $\mathcal{L}$ is a linear differential operator of order $d$, mapping a univariate function $u:[-1,1]\rightarrow \mathbb{R}$ to another function 
$Lu:[-1,1]\rightarrow \mathbb{R}$, and $b_i(u,u', \dots, u^{(d)}, \lambda)=0$ for $i=1,\ldots,d$, are homogeneous boundary conditions (which we allow to depend affinely on $\lambda$). We assume that the spectrum of~\eqnref{eq:diffeq} is discrete and that all eigenvalues have finite algebraic multiplicity. The goal is to compute the (selected) eigenvalues $\lambda$ and their corresponding eigenfunctions $u$. 

Operator eigenvalue problems have a long history. Classical problems include the Schr\"odinger equation in quantum mechanics~\cite{griffiths2018introduction}, the Sturm-Liouville equation~\cite{folland1992fourier}, and the Orr-Sommerfeld equation in fluid dynamics~\cite{orszag1971accurate}. A number of further examples can be found in~\cite{trefethen2017exploring,trefethenEmbree}, which are illustrated in~\cite{trefethen2017exploring} using Chebfun~\cite{Chebfun14}.

In this paper we develop an algorithm for solving operator eigenvalue problems~\eqref{eq:diffeq}. The focus is on spectral methods, which are one of the main computational frameworks for solving ODEs of the form $\mathcal{L}u=f$ or~\eqnref{eq:diffeq}. Broadly speaking, the guiding principle for these methods is as follows.
\begin{enumerate}
\item Form a set of basis functions $\{ u_1,\ldots,u_n\}$, where each $u_i:[-1,1]\rightarrow \mathbb{R}$. 
\item Find an approximate solution $u = \sum_{i=1}^n c_i u_i \in \quasi{U} := \mbox{span}\{ u_1,\ldots,u_n\}$.

\end{enumerate}
Spectral methods differ with respect to how the basis functions are chosen and how the approximate solution $u\in \quasi{U}$ is determined. A typical choice of $u_i$ is the Chebyshev polynomial $T_{i-1}$, or a Fourier (trigonometric) polynomial.\ The rationale is that if the solution is smooth, then it can be approximated by such polynomials efficiently; for example one gets exponential convergence for analytic solutions~\cite{trefethenatap}. Methods with this property are called spectrally accurate. We list some variants of spectral methods below. We describe the methods both for operator equations (ODEs) $\mathcal{L}u=f$ and eigenvalue problems~\eqref{eq:diffeq}. 

\begin{itemize}
\item Spectral Galerkin methods: Choose $u_i$ to satisfy the boundary conditions. 
Then choose $c_i$ by enforcing that the residual is orthogonal
to $\mbox{span}(\quasi{V})$ (usually $\quasi{V} = \quasi{U}$\footnote{We follow a common abuse of notation that allows a matrix to represent its span too. Here, $\quasi{U}$ or $\quasi{V}$ may denote a quasimatrix or the space spanned by its columns; see comments near the end of this section.}). For ODEs this becomes $\ip{v}{\mathcal{L}u-f}=0$ for all $v\in \quasi{V}$. For eigenvalue problems, the condition is $\ip{v}{\mathcal{L}u-\lambda u}=0$. Here and below,  the inner product is often the $L_2$-inner product $\ip{u}{v} = \int_{-1}^1 u(x) v(x)\ dx$. 
  \item tau method: Choose $u_i$ as Chebyhev polynomials\footnote{Legendre polynomials have also been used in the literature \cite{Canuto2006}, although to a lesser extent.} and apply boundary conditions as side constraints (boundary bordering). Then choose $c_i$ via imposing that the residual is a linear combination of high-degree Chebyshev polynomials. For ODEs this means $\mathcal{L}u-f=\sum_{i\geq n-d}\tau_i T_i$ and for eigenvalue problems it is equivalent to $\mathcal{L}u-\lambda u=\sum_{i\geq n-d}\tau_i T_i$.  

\item Spectral collocation methods~\cite{trefethen2000spectral}: Choose $u_i$ to be smooth functions, usually Chebyshev or Fourier polynomials. Choose ${\bf c}=[c_1,\ldots,c_n]^T$ (the solution $u$ is usually represented via the values at $x_i$; mathematically this is equivalent to choosing ${\bf c}$) by forcing the equation to hold ($\mathcal{L}u-f=0$ or $\mathcal{L}u-\lambda u=0$) at these points $(x_i)_{i=0}^n$, called collocation points. Impose boundary conditions by replacing rows of the equation. Chebfun's default ODE solver is based on a collocation approach~\cite{aurentz17, driscoll2016rectangular, trefethen2017exploring}. 
  \end{itemize}
\vspace{3mm}

Spectral methods are used in Chebfun for solving both operator equations and operator eigenvalue problems to high accuracy.\ Chebfun exploits block structure information about the corresponding operators at the continuous level. Such information includes the index, nullity and deficiency of the operator, whose connection to the amount of rectangularity of the operator is made precise in~\cite{aurentz17}. It is this rectangularity that clarifies the right size of the rectangular discretization matrix~\cite{driscoll2016rectangular} to be used in numerical computation. 

In this paper we consider a framework in which, conceptually, we would like to 
\[
\minimize_{u\in \quasi{U}} \left\|
\begin{bmatrix}
  \mbox{residual in operator equation with $u$}\\
  \mbox{residual in boundary conditions with $u$}
\end{bmatrix}
\right\|_2.
\]
Here the norm $\|\cdot\|_2$ is applied to an object of the form $
\begin{bmatrix}
u\\ {\bf b}
\end{bmatrix}=
\begin{bmatrix}
\mbox{function}\\  \mbox{vector}
\end{bmatrix}
$. 
We define this by\footnote{Different weightings can be considered, i.e., $\left\|\begin{bmatrix}u\\ {\bf b}\end{bmatrix}\right\|_2 = \sqrt{\alpha\|u\|_{L_2}^2+\beta\|{\bf b}\|_2^2}$ for some $\alpha,\beta>0$. We will revisit this in Section~\ref{sec:bcexact}.
} $\left\|\begin{bmatrix}u\\ {\bf b}\end{bmatrix}\right\|_2 = \sqrt{\|u\|_{L_2}^2+\|{\bf b}\|_2^2} = \sqrt{\int_{-1}^1|u(x)|^2\ dx + \|{\bf b}\|_2^2}$. 
Specifically, we focus on~\eqref{eq:diffeq} and develop spectral methods which proceed as follows. Start with (or keep building) a set of basis functions $\{\phi_1,\ldots,\phi_n\}$, and consider finding the function in $\quasi{U} := \mbox{span} \{\phi_1,\ldots,\phi_n\}$ such that the residual is minimized, that is,
\begin{equation} 
 \label{eq:lsprob}
\minimize_{u\in \quasi{U},\lambda\atop \|u\|_2=1} \left\|
\begin{bmatrix}
\mathcal{L} u - \lambda u  \\
b_1(u,u', \dots, u^{(d)}, \lambda)\\
\vdots\\
b_d(u,u', \dots, u^{(d)}, \lambda)
\end{bmatrix}
\right\|_2. 
\end{equation}
Note that the boundary conditions are part of the objective function; thus they are enforced approximately. This lets us start with any set of basis functions $\quasi{U}$, unlike the traditional Galerkin method where the basis functions are usually chosen a priori to satisfy the boundary conditions~\cite{gottlieb1977numerical,hesthaven2007spectral}, making the method inapplicable when the boundary conditions depend on $\lambda$. If it is important to impose the boundary conditions strictly, we can apply a large weight to the last $d$ components of the objective function. We also describe imposing them exactly in Section~\ref{sec:bcexact}. 

In order to solve~\eqref{eq:lsprob} we formulate it as a
`infinite+finite' dimensional, rectangular eigenvalue problem. 
We then solve it using a generalization of the algorithm by Ito and Murota~\cite{itomurota2016}, originally designed for discrete rectangular eigenvalue problems. 

Our proposed algorithm falls in the category of spectral methods. It has the following properties: 
\begin{itemize}
\item No conditions are imposed on the basis functions $u_1,\ldots,u_n$, giving the flexibility to choose them freely. 
\item The least-squares formulation directly minimizes the residual. This makes the convergence analysis trivial in some sense: the convergence speed is precisely the speed at which the basis functions are able to approximate the solution. 
\end{itemize}

Our method has the property that it allows  $u_i$ to be chosen arbitrarily, yet it has a mixture of a Galerkin (in the sense that the ``joint'' residual is orthogonal to a well-chosen subspace) and a collocation flavor (where the collocation points are chosen to be the entire domain; of course the operator equation cannot be imposed exactly everywhere, so instead we find the  least-squares fit).

\subsection{Contributions}
Below we list some of the advantages of our least-squares approach.

\paragraph{Flexibility: nonstandard basis functions}
We believe the most significant advantage of our approach is the flexibility in terms of the choice of the basis functions. In classical approaches---either collocation or Galerkin---the basis functions are chosen a priori, usually as degree-graded polynomials or Fourier series. Such expansion is often very effective, and convergence is exponential provided that the solution is analytic in the domain. However, when the solution is  not smooth, these methods converge slowly. 
By contrast, in our approach the basis functions are allowed to be arbitrary. In particular, when prior knowledge about the solution is available, such as the location of a singularity or solutions for nearby problems, then by incorporating tailor-made functions into the basis functions we can dramatically accelerate the convergence, as we illustrate through numerical examples. 

\paragraph{Flexibility: $\lambda$-dependent boundary conditions}
Our formulation also naturally allows for exotic boundary conditions, in particular, those that depend affinely on the eigenvalue $\lambda$, that is, the boundary conditions are of the form $b_i(u,u', \dots, u^{(d)}, \lambda)=\lambda (\sum_{i=1}^{d}c_iu^{(d)})+(\sum_{i=1}^{d}\tilde c_iu^{(d)}) = 0$ where $c_i,\tilde c_i\in\mathbb{C}$. An example is the indefinite problem $-u'' = \lambda u$ on $[0, \pi/2]$ with boundary conditions $u'(0) = \lambda \big( \frac{3}{2} u(0) + u'(0)\big)$, and $u'(\frac{\pi}{2}) = 0$\footnote{This particular problem is studied in~\cite[p. 265]{Pryce93}, which shows that it has three eigenfunctions that have no roots on $(0, \pi/2)$.}. Another example is $u^{(4)} = \lambda u''$ with boundary conditions $u(0) = u'(0) = 0$, $u''(1) = 0$, $u^{(3)}(1) - \lambda \gamma u'(1) = 0$ where $\gamma \in [0, 1]$. This problem arises in the study of critical loads for divergence of a clamped-free elastic bar; see~\cite{Marletta03} and the references therein.

We shall show that it is straightforward to solve such problems using our framework. While it is not impossible to do so with classical methods, with our approach one can deal with such problems with minimal modification. 

\paragraph{Flexibility: generalized eigenvalue problems}
Our approach also make it straightforward to treat \emph{generalized} operator eigenvalue problems of the form 
\begin{equation}
  \label{eq:diffeq:gep}
\mathcal{L}_A u = \lambda\mathcal{L}_Bu, 
\end{equation}
where $\mathcal{L}_A$ and $\mathcal{L}_B$ are linear differential operators  of orders $d$ and $p$, respectively. Again the boundary conditions may be $\lambda$-dependent. 

\paragraph{Miscellaneous}
Our least-squares approach also allows for variants that can often be formulated in a straightforward fashion. As an example, we discuss a reformulation of the differential equation as an integral equation, which is known (as shown in Greengard's papers, e.g.~\cite{greengard1991spectral}) to often reduce the condition number of the  discretized problem  and therefore provide more accurate solutions than standard approaches for differential equations. This idea was revisited by Driscoll using Chebfun~\cite{driscoll2010automatic}. We observe that the same holds also for our algorithm.

\paragraph{Disadvantages}
Having described the advantages of our least-squares approach, it is important to also highlight the drawbacks of our approach. Most notably, it relies heavily on one's ability to work with and manipulate quasimatrices. Such operations are conveniently offered by Chebfun (and indeed the essence of our algorithms can be implemented in a few lines of MATLAB code), but otherwise (on other platforms) the algorithm can be nontrivial to implement (to emulate the method approximately, one can take many sample points in the domain and perform a least-squares fit). 
In addition, for ``easy'' problems with smooth solutions, classical algorithms work exceptionally well, with exponential convergence and fast solution of the discretized linear algebra problem often available. Our approach will not be more efficient in such settings (e.g. when one takes $U$ to be the same polynomial basis.) 

This paper is organized as follows.\ In Section~\ref{sec:LSode} we describe our algorithm LSode for ODEs. In Section~\ref{sec:ode} we explain our formulation for ODE eigenvalue problems. Section~\ref{sec:svd} describes the SVD of objects containing quasimatrices and matrices, which is needed for our eigenvalue algorithm. We then introduce our algorithm LSeig in Section~\ref{sec:LSeig}. In Section~\ref{sec:bc} we discuss issues related to boundary conditions, namely imposing them exactly, and nonstandard conditions involving $\lambda$.\ Section~\ref{sec:integral} treats reformulations into an integral equation and its adaptation to LSeig.\ We present numerical experiments in Section~\ref{sec:exp} to illustrate the performance of LSeig. 

\emph{Notation}. 
We denote by $u'$, $u''$ the first and second derivative of a function $u$, and more generally $u^{(d)}$ denotes the $d$-th derivative. We use lower-case boldface letters for vectors, simple capital letters for matrices and calligraphic letters for operators. 

We denote quasimatrices by sans-serif capital letters, e.g. $\quasi{A}$. A (column) quasimatrix $\quasi{A}$ is an $\infty \times n$ `matrix' whose columns $a_j$ are functions rather than discrete vectors. This means that the first index of a rectangular matrix becomes continuous while the second one remains discrete. Thus, $\quasi{A}$ represents a linear transformation from $\Cz^{n}$ to $L^2 [a,b]$. 
We mainly treat problems on $[-1,1]$ for concreteness, but the extension to $[a,b]$ is straightforward.
The term quasimatrix was coined by Stewart \cite[p. 33]{stewart98}. Quasimatrices are discussed by de Boor~\cite{deBoor91} and Trefethen and Bau \cite[pp. 52-54]{trefethenNLAbook}. They became part of computational practice with the introduction of Chebfun in 2004 \cite{battles05, battles04, trefethen2013householder}. 

In what follows, $\quasi{U}=[u_1 \dots,u_n] \in \mathbb{R}^{\infty\times n}$ denotes a quasimatrix~\cite{Chebfun14} whose columns are the basis functions, whose choice is made a priori and often nonstandard and problem-dependent (one can take $u_i$ to be polynomials but the strength of our approach is exhibited with other choices; we shall discuss several specific choices for specific problems). 

We will make extensive use of decompositions of quasimatrices, in particular the QR factorization and SVD.\ For details on these see~\cite{townsend15,trefethen2013householder} and~\cite[chap. 6]{Chebfun14}. Almost all operations with matrices have counterparts for quasimatrices. For example for a function $u$ and a quasimatrix $\quasi{Q}=[q_1,\ldots,q_n]\in\mathbb{C}^{\infty \times n}$, we let $\quasi{Q}^Tf$ denote the $n\times 1$ vector whose $i$th element is $\ip{q_i}{f}$. Moreover, a nonstandard (but straightforward) object that we need to deal with extensively is a ``quasimatrix-matrix'' of the form 
$\begin{bmatrix}
  \quasi{U} \\ B\end{bmatrix}$, where $\quasi{U}\in\mathbb{C}^{\infty \times n}$ is a quasimatrix and $B$ is a $d\times n$ matrix. 
We say that such an object is of size\footnote{Perhaps it is more accurate to say they are of size $([-1,1]+d) \times n$. We use $(\infty+d) \times n$ for brevity.} ``$(\infty+d) \times n$''.
These are an instance of block operators introduced in~\cite{aurentz17}. 
We use standard operations with such objects, including multiplication by a vector ${\bf v}\in\mathbb{C}^{n}$ as defied by the formula
$\begin{bmatrix}
  \quasi{U} \\ B\end{bmatrix} {\bf v} =
\begin{bmatrix}
  \quasi{U} {\bf v}\\B{\bf v}
\end{bmatrix}
$, which yields a ``function-vector'' object.

\section{Least-squares formulation for ODE $\mathbf{\mathcal{L}u=f}$} \label{sec:LSode}
Our emphasis is on operator eigenvalue problems, but for a gentle introduction we first discuss solving differential equations $\mathcal{L}u=f$ with boundary conditions $b_i(u,u',\ldots,u^{(d)})=f_{b_i}$ for $i=1,\ldots,d$, where $f_{b_i}\in\mathbb{R}$ and $b_i$ is a linear function of its arguments. 

The main idea is to find $u\in U$ such that 
\begin{equation}
  \label{eq:goal1}
\minimize_{u\in U}\left\|
  \begin{matrix}
\mathcal{L}u-f  \\
 b_1(u,\ldots,u^{(d)})-f_{b_1}   \\
\vdots\\
 b_d(u,\ldots,u^{(d)})-f_{b_d}      
  \end{matrix}
\right\|_2, 
\end{equation}
that is, the norm  is minimized of the function-vector consisting of the residual of the operator and boundary conditions.

\paragraph{Basis functions satisfying the boundary conditions}
For simplicity first consider the case where the basis functions $u_j$ satisfy the boundary conditions $b_i(u_j,u_j',\ldots)=0$. Then 
the bottom $d$ elements are 0 and writing the solution $u=\sum_{i=1}^nc_iu_i$,~\eqref{eq:goal1} reduces to 
\[
\minimize_{u\in U}\|\mathcal{L}u-f\|_2 = \minimize_{{\bf c}\in \mathbb{C}^{n}}\|(\mathcal{L} \quasi{U}) {\bf c} - f\|_2. 
\]
This is a least-squares problem involving a quasimatrix. As with a matrix least-squares problem, it can be solved via the quasimatrix QR factorization\footnote{This is a ``thin'' QR factorization. For quasimatrices it serves little purpose to think of the ``full'' QR, as the functions live in an infinite-dimensional space.} $\mathcal{L} \quasi{U} = \quasi{Q} R$, where $\quasi{Q}\in\mathbb{C}^{\infty \times n}$ 
is orthonormal $\quasi{Q}^T \quasi{Q} = I_n$ (the $(i,j)$ entry of $\quasi{Q}^T \quasi{Q}$ is $\ip{q_i}{q_j}$) and $R\in\mathbb{C}^{n \times n}$ is upper-triangular~\cite{trefethen2013householder}, as ${\bf c}=R^{-1} \quasi{Q}^T f$. 
The formulation is mathematically (but not numerically) equivalent to Galerkin with the choice $\quasi{V}:= \mathcal{L} \quasi{U}$. 

\paragraph{General basis functions}
More generally, when the basis functions $u_j$ do not satisfy the boundary conditions, we rewrite~\eqref{eq:goal1} as  
finding the minimizer for 
\begin{equation}
  \label{eq:goal2}
\min_{{\bf c}\in \mathbb{C}^n}\left\|
  \begin{matrix}
(\mathcal{L} \quasi{U}) {\bf c}-f  \\
\footnotesize b_1(\quasi{U} {\bf c},(\quasi{U} {\bf c})',\ldots,(\quasi{U} {\bf c})^{(d)}) -f_{b_1}   \\
\vdots\\
 b_d(\quasi{U}{\bf c},(\quasi{U} {\bf c})',\ldots,(\quasi{U} {\bf c})^{(d)})  -f_{b_d}
\normalsize
  \end{matrix}
\right\|_2
=\min_{{\bf c}\in \mathbb{C}^n}\left\|
  \begin{matrix}
    \begin{bmatrix}
\mathcal{L} \quasi{U}\\
B      
    \end{bmatrix} {\bf c}
-
\begin{bmatrix}
f  \\
{\bf f}_b
\end{bmatrix}
  \end{matrix}
\right\|_2
\end{equation}
where $B\in\mathbb{C}^{d\times n}$ and ${\bf b}\in\mathbb{C}^d$. 

The problem~\eqref{eq:goal2} is still solvable, once one has the QR factorization for 
the quasimatrix-matrix 
$\begin{bmatrix}
\mathcal{L} \quasi{U}\\ 
B      
\end{bmatrix}$, which can be obtained as follows. We first compute the quasimatrix QR factorization $\mathcal{L} \quasi{U} = \quasi{Q}_L R_L$ \cite{trefethen2013householder} and the standard matrix QR factorization $B=Q_B R_B$, so that
\begin{equation}
  \label{eq:QRobj}
\begin{bmatrix}
  \mathcal{L}\quasi{U} \\B
\end{bmatrix} = 
\begin{bmatrix}
  \quasi{Q}_L R_L\\Q_B R_B
\end{bmatrix}  =
\begin{bmatrix}
  \quasi{Q}_L& \quasi{0} \\0  & Q_B
\end{bmatrix} 
\begin{bmatrix}
  R_L\\R_B
\end{bmatrix}.  
\end{equation}
Note that $\quasi{\tilde Q}:=\begin{bmatrix}  \quasi{Q}_L& \quasi{0} \\0  & Q_B\end{bmatrix}\in\mathbb{C}^{(\infty+d)\times 2n} $ is an orthonormal quasimatrix-matrix satisfying $\quasi{\tilde Q}^T \quasi{\tilde Q} = I_{2n}$. (For the quasimatrix-matrix $\quasi{\tilde Q}$ we use the same sans-serif font as quasimatices.) One can complete the QR factorization via another thin QR 
$\begin{bmatrix}
  R_L\\R_B 
\end{bmatrix} = \tilde QR, $
giving 
\begin{equation}
  \label{eq:LUBqr}
\begin{bmatrix}
  \mathcal{L} \quasi{U} \\B
\end{bmatrix} =
\left(
\begin{bmatrix}
  \quasi{Q}_L & \quasi{0} \\0  & Q_B
\end{bmatrix} 
\tilde Q
\right)
R
= \quasi{Q} R.
\end{equation}

We summarize the algorithm below. 
\begin{algorithm}[h!]
\begin{algorithmic}[1]
\STATE Form the quasimatrix $\mathcal{L} \quasi{U}$. 
\STATE Represent boundary conditions as $B{\bf c}= {\bf f_b}$, where 
${\bf c}=[c_1,\ldots,c_n]^T, u=\sum  c_i u_i$.
\STATE Compute QR factorization of the $(\infty+d)\times n$ quasimatrix-matrix 
(as in~\eqref{eq:LUBqr})
\begin{equation}
  \label{eq:QRobj2}
\begin{bmatrix}
  \mathcal{L} \quasi{U} \\B
\end{bmatrix} = \quasi{Q} R. 
\end{equation}
\STATE Obtain solution $u=\quasi{U} {\bf c}$ where ${\bf c}=R^{-1}(\quasi{Q}^*
\begin{bmatrix}
f\\ {\bf f_b}  
\end{bmatrix}
)$. 
\end{algorithmic}
\caption{LSode: Least-squares method for  $\mathcal{L} u = f$, and boundary conditions $b_i(\quasi{U}{\bf c},(\quasi{U} {\bf c})',\ldots,(\quasi{U} {\bf c})^{(d)}) = f_{b_i}$ for $i=1,\ldots,d$.
$\quasi{U}:=[u_1,\ldots,u_n]$ are the user-defined basis functions.} 
\label{LSode:alg}
\end{algorithm}

The features of this least-squares formulation include
\begin{enumerate}
\item The least-squares residual represents exactly the residual in the problem, including the boundary condition. 
\item If desired, the boundary conditions can be imposed exactly (see Section~\ref{sec:bcexact}). 
\item It has the flavor of spectral collocation methods wherein the collocation points are taken to vary continuously. It also has the Galerkin flavor, as the solution can be characterized equivalently as follows: the augmented residual $\begin{bmatrix}  \mathcal{L}u-f \\ B {\bf c} - {\bf f_b}\end{bmatrix}$ is orthogonal to $\quasi{Q}$.
\end{enumerate}
These features carry over to the eigenvalue algorithm we develop in the next section. 

We believe LSode, while conceptually straightforward, is a new algorithm:\ Boyd ~\cite[\S~3.1]{Boyd2001} briefly discusses a least-squares approach for ODEs, but suggests solving it using the normal equation (which is numerically unstable) and without addressing boundary conditions.

Let us note, in passing, that there is an attractive class of finite element methods, called Least-Squares Finite Element Methods (LSFEM), for the numerical solution of PDEs. These methods minimize the residual in the differential equation together with the residual in the boundary conditions. These methods have been motivated by the desire to recover, in general settings, the advantageous features of Rayleigh-Ritz methods; see~\cite{Bochev09} and the references therein.\ LSode is in some sense a spectral analogue of LSFEM for ODEs.

\subsection{Numerical illustration}\label{sec:LSODEex}
To illustrate the advantages offered by the flexibility of basis functions in LSode, here we 
consider building the basis functions $u_1,\ldots,u_n$ by solving a nearby ODE (which might have been solved or is easier to solve). 
Namely, consider the ``pilot'' ODE 
$\mathcal{L}_1u = f$, where $f(x)=\mbox{exp}(x)$ and 
\begin{equation}
  \label{eq:ODE1}
  \mathcal{L}_1u = u''+|x|u' ,\qquad u(\pm 1) = 0. 
\end{equation}
We use solutions for~\eqref{eq:ODE1} as the basis functions $u_i$ for solving the slightly different ODE $\mathcal{L}_2u = f$, where $\mathcal{L}_2$ is related but not equal to $\mathcal{L}_1$: 
\begin{equation}
  \label{eq:ODE2}
  \mathcal{L}_2u = u''+|x|u'+u ,\qquad u(\pm 1) = 0. 
\end{equation}
Specifically, here we used $\mathcal{L}_1$ (or its inverse) to form a Krylov subspace (informally, this is the subspace $\mbox{span}(\mathcal{L}^{-1}f,\mathcal{L}^{-2}f,\ldots,\mathcal{L}^{-n}f)$, orthogonalized by the Arnoldi process; see~\cite{gilles2019continuous} for more discussion on Krylov subspaces for differential operators). We then let $\quasi{U}=[u_1,\ldots,u_n]$ be the orthonormal basis and invoke LSode. We compare its performance with the standard choice of $u_i$, global polynomials $u_i = T_{i-1}(x)$. The results are shown in Figure~\ref{fig:nearby}. 

\begin{figure}[htbp]
  \begin{minipage}[t]{0.495\hsize}
\includegraphics[width=0.9\textwidth]{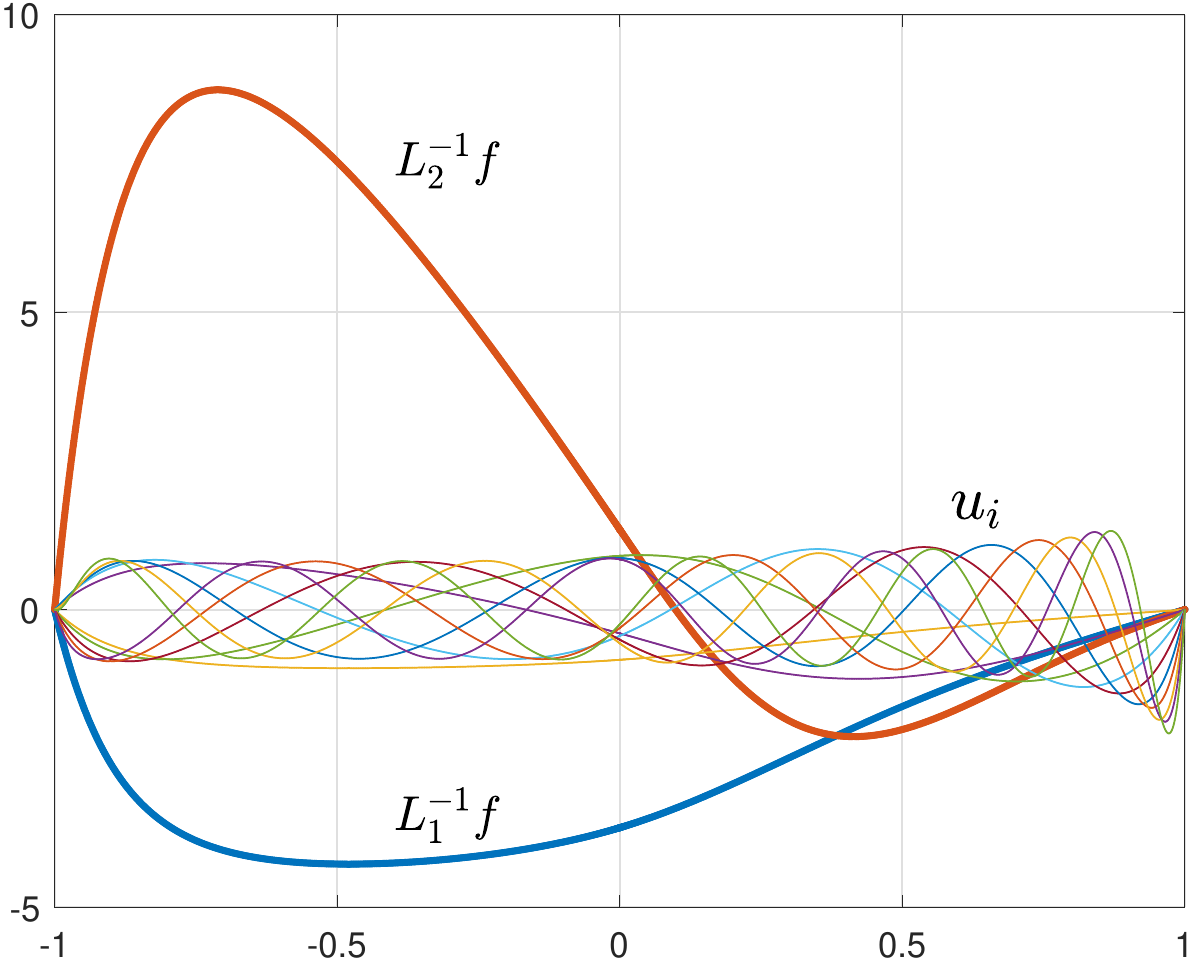}
  \end{minipage}   
  \begin{minipage}[t]{0.495\hsize}
\includegraphics[width=0.95\textwidth]{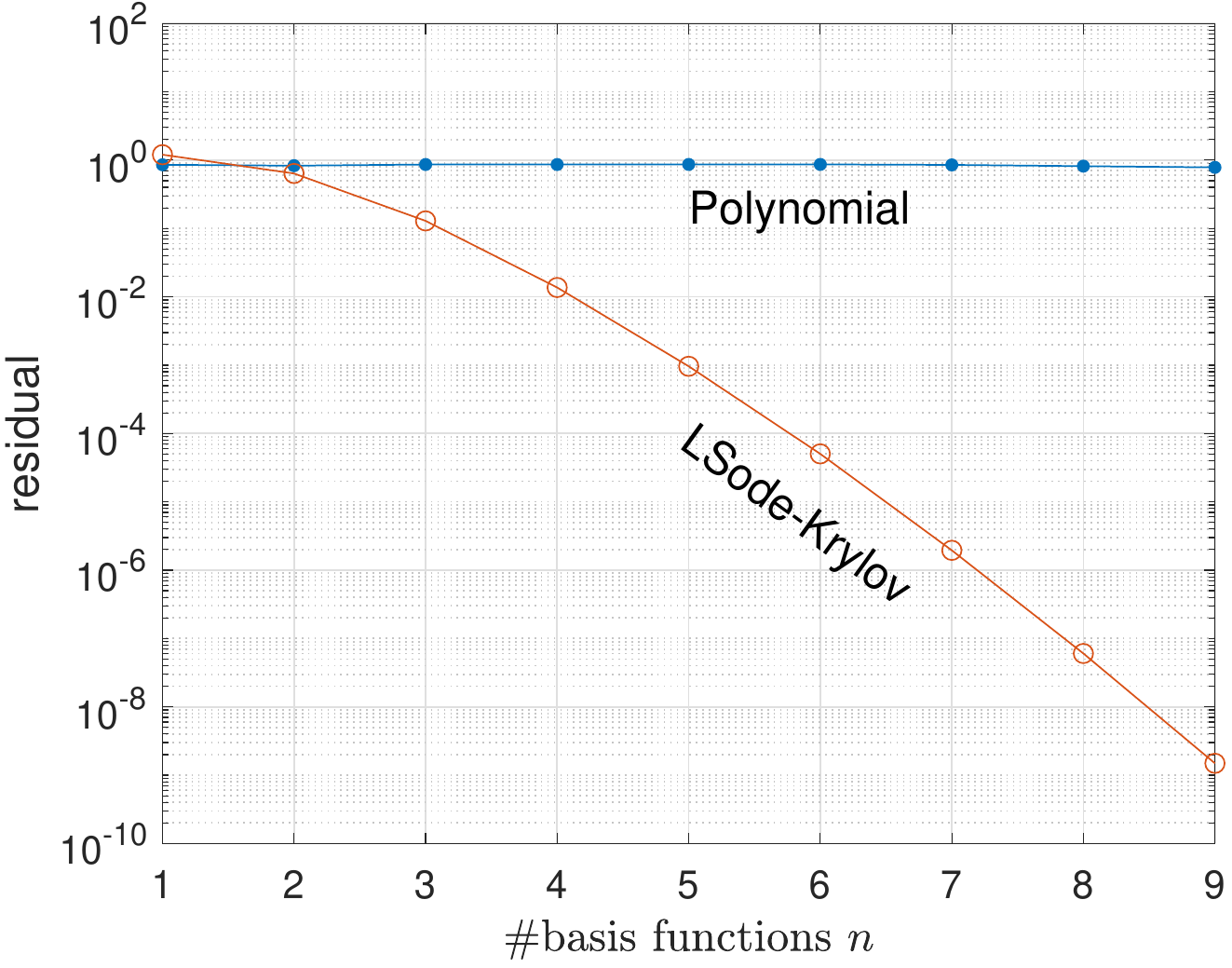}
  \end{minipage}
  \caption{
Illustration of choice of basis functions for solving~\eqref{eq:ODE2} with LSode: 
Polynomials vs. Krylov basis functions obtained by solving nearby ODEs~\eqref{eq:ODE1}. 
Left: solutions $L_1^{-1}f,L_2^{-1}f$ together with the basis functions $u_i$. Note that 
$L_1^{-1}f$ and $L_2^{-1}f$ are themselves not very similar, yet a Krylov basis w.r.t. $L_1^{-1}$ offers a powerful subspace in which to find $L_2^{-1}f$.}
\label{fig:nearby}
\end{figure}

Figure~\ref{fig:nearby} demonstrates that by using specially crafted basis functions, one can obtain much faster convergence than standard bases. The flexibility offered by the least-squares framework allows us to easily explore nonstandard basis functions. 

Note that since by definition LSode minimizes the residual (together with the boundary conditions), no standard collocation or coefficient method based on global polynomials can do better than the LSode with polynomials shown in the figure. 

Here we considered problems with the $|x|$ term in~\eqref{eq:ODE1},~\eqref{eq:ODE2} so that they are not easily solvable by a classical approach. The idea of using solutions for nearby problems appears to work well in a variety of settings (a detailed study is left for future work). Clearly, once $\quasi{U}$ is obtained, one can consider solving a host of ODEs similar to~\eqref{eq:ODE1} using the same $\quasi{U}$; this is much more efficient than solving each problem separately.

\section{Least-squares formulation for $\mathbf{\mathcal{L}u=\lambda u}$}\label{sec:ode}
Now we turn to the main topic of operator eigenvalue problems  $\mathcal{L}u=\lambda u$, with boundary conditions $b_i(u,u',\ldots,u^{(d)})=0$ for  $i=1,\ldots,d$. 

Again the core idea is to find $u\in U$ such that the operator and boundary condition residuals are minimized simultaneously:
\begin{equation}
  \label{eq:goal1eig}
\minimize_{u\in U}\left\|
  \begin{matrix}
\mathcal{L}u-\lambda u  \\
 b_1(u,\ldots,u^{(d)})    \\
\vdots\\
 b_d(u,\ldots,u^{(d)})    \\
  \end{matrix}
\right\|_2.
\end{equation}

\paragraph{Basis functions satisfying the boundary conditions}
If each of the basis functions $u_j$ satisfy the boundary conditions, that is, $b_i(u_j,u_j',\ldots)=0$ for $i=1,\ldots,d$, then~\eqref{eq:goal1eig} reduces to (writing as before the solution $u=\sum_{i=1}^nc_iu_i$)
\[
\min_{u\in U}\|\mathcal{L}u-\lambda u\|_2=
\min_{{\bf c}\in \mathbb{C}^{n}}\|(\mathcal{L}\quasi{U}) {\bf c}-\lambda \quasi{U} {\bf c}\|_2. 
\]
This is a least-squares eigenvalue problem involving quasimatrices. 

At first sight this is a difficult problem: supposing there is a solution with zero residual 
$(\mathcal{L} \quasi{U}) {\bf c} = \lambda \quasi{U} {\bf c}$, it is an $\infty\times n$ generalized eigenvalue problem. A matrix analogue would be an $m\times n$ ($m>n$) \emph{rectangular} eigenvalue problem $A {\bf x}=\lambda B {\bf x}$ for 
$A,B\in\mathbb{C}^{m\times n}$ \cite{Wright02}.
Such problems are not classically treated in numerical linear algebra, and they usually do not have solutions satisfying the equation exactly. Nonetheless, several practical methods for rectangular eigenvalue problems have been proposed, including \cite{das2013solving, Li21, Li20, Li21NLAA}. 

In this paper we focus on the algorithm by Ito and Murota~\cite{itomurota2016}, which is easy to implement and finds solutions that are optimal in a certain sense (it finds the smallest perturbations of $A,B$ such that the equation has $n$ solutions). 
The Ito-Murota algorithm relies on an SVD to reduce the problem to a square $n\times n$ generalized eigenvalue problem. We review the algorithm in Section~\ref{sec:IM}. 

Fortunately, this algorithm can be generalized readily to quasimatrices, and then in turn to quasimatrix-matrix objects. The crucial tool is the SVD of such objects, which we discuss in Section~\ref{sec:svd}. 

\paragraph{General basis functions}
When the basis functions $u_j$ do not satisfy the boundary conditions $b_i(u_j,u_j',\ldots)\neq 0$, the goal is to find the minimizer for 
\begin{equation}
  \label{eq:goal2eig}
\min_{{\bf c}\in \mathbb{C}^n}\left\|
  \begin{matrix}
(\mathcal{L}\quasi{U}){\bf c} - \lambda \quasi{U} {\bf c}  \\
\footnotesize b_1(\quasi{U}{\bf c},(\quasi{U} {\bf c})',\ldots,(\quasi{U} {\bf c})^{(d)})    \\
\vdots\\
 b_d(\quasi{U} {\bf c},(\quasi{U} {\bf c})',\ldots,(\quasi{U} {\bf c})^{(d)})    
\normalsize
  \end{matrix}
\right\|_2
=\min_{{\bf c}\in \mathbb{C}^n}\left\|
  \begin{matrix}
    \begin{bmatrix}
\mathcal{L} \quasi{U}\\
B      
    \end{bmatrix} {\bf c}
-\lambda 
\begin{bmatrix}
\quasi{U}  \\
0
\end{bmatrix} {\bf c}
  \end{matrix}
\right\|_2
\end{equation}
where $B\in\mathbb{C}^{d\times (n+1)}$. This is a rectangular eigenvalue problem for quasimatrix-matrix objects. We show that this can also be solved using the Ito-Murota method, via the SVD for quasimatrix-matrix objects. We now turn to computing such an SVD.

\section{SVD of objects involving quasimatrices and matrices}\label{sec:svd}
In this section we explain how to compute the SVD of different objects required in our least-squares framework for operator eigenvalue problems. We first review the simplest case of computing the SVD of a quasimatrix and then proceed with more involved objects. 

Let $\quasi{A}$ be an $\infty \times n$ quasimatrix. Following \cite{trefethen2013householder}, the SVD of $\quasi{A}$ can be  computed 
in two steps. We first compute the QR decomposition $\quasi{A} = \quasi{Q_A} R_A$, where $\quasi{Q}_A$ is an $\infty \times n$ quasimatrix with orthogonal columns and $R_A$ is an $n \times n$ upper-triangular matrix. We then proceed with computing the SVD of 
\begin{equation}
\label{eq:svdRfactor}
R_A = U_R \Sigma V^\ast
\end{equation}
 resulting in the SVD
\[
\quasi{A} = \underbrace{(\quasi{Q}_A U_R)}_{\quasi{U}} \Sigma V^\ast,
 \]
where $\quasi{U}$ is an $\infty \times n$ quasimatrix with orthogonal columns, $\Sigma$ is $n \times n$ diagonal with singular values arranged in non-increasing order and $V$ is an $n \times n$ unitary matrix. 

We will also need the SVD of the horizontal concatenation of two quasimatrices. Let $[\quasi{A}\  \quasi{B}]$ be an $\infty \times 2n$ quasimatrix. Analogously to the previous case, we first compute the QR decomposition $[\quasi{A} \ \quasi{B}]= \quasi{Q} R$ where $\quasi{Q}$ is an $\infty \times 2n$ quasimatrix with orthogonal columns and then compute the SVD of the $2n \times 2n$ upper-triangular matrix $R = U_R \Sigma V^\ast$. Partitioning we have
\[
U_R =
  [\underbrace{U_{R1}}_{2n \times n}\ \underbrace{U_{R2}}_{2n \times n}], \ \
\Sigma = 
\begin{bmatrix}
\Sigma_{1:n}&\\
& \Sigma_{n+1:2n}  
\end{bmatrix}, \ \
V = 
\begin{bmatrix}
V_{11}& V_{12}\\
V_{21}& V_{22}
\end{bmatrix}.
\]
Defining $\quasi{U}_1 = \quasi{Q} U_{R1}$ and $\quasi{U}_2 = \quasi{Q} U_{R2}$, note that these  $\infty \times n$ qusimatrices have orthogonal columns.
Thus we obtain the SVD
\begin{equation} 
\label{eq:svdQuasiAB}
\begin{bmatrix}
\quasi{A} & \quasi{B}
\end{bmatrix}
= \quasi{U} \Sigma V^{\ast} = 
\underbrace{\begin{bmatrix}
\quasi{U}_1 & \quasi{U}_2
\end{bmatrix}}_{\quasi{U}}
\begin{bmatrix}
\Sigma_{1:n}&\\
& \Sigma_{n+1:2n}  
\end{bmatrix} 
\begin{bmatrix}
V_{11}^\ast& V_{21}^\ast\\
V_{12}^\ast& V_{22}^\ast
\end{bmatrix},
\end{equation}

Next, we discuss the computation of the SVD of a quasimatrix-matrix object which is the vertical concatenation of an $\infty \times n$ qusimatrix $\quasi{A}$ and a $d \times n$ matrix $C$. Starting from two QR factorizations we have
\[
\begin{bmatrix}
  \quasi{A}\\C
\end{bmatrix} = \begin{bmatrix}
  \quasi{Q}_A R_A\\Q_C R_C
\end{bmatrix}  = \begin{bmatrix}
\quasi{Q}_A & 0 \\ 0 & Q_C
\end{bmatrix} \begin{bmatrix}
R_A \\ R_C
\end{bmatrix}.
\]
We then compute the thin SVD of the $(n+d) \times n$ matrix
\begin{equation}
\begin{bmatrix}
\label{svd_R_concat:eq}
  R_A\\R_C
\end{bmatrix} = U_R \Sigma V^\ast, 
\end{equation}
resulting in the SVD
\begin{equation}
\label{eq:svdRfactorquasimat-mat}
\begin{bmatrix}
  \quasi{A}\\C
\end{bmatrix} =
\underbrace{\left(\begin{bmatrix}
  \quasi{Q}_A& 0 \\ 0 & Q_C
\end{bmatrix} 
U_R\right)}_{\quasi{U}} \Sigma V^*.
\end{equation}
Here, $\quasi{U}$ is an $(\infty + d) \times n$ quasimatrix-matrix. It can be easily checked that $\quasi{U}^* \quasi{U} = I_{n}$ and $V^*V =I_n$.

Finally, consider objects of the form 
$\big[
\begin{smallmatrix}
  \quasi{A} & \quasi{B}\\
  C & D  
\end{smallmatrix}
\big]$, 
where $\quasi{A}$ and $\quasi{B}$ are $\infty \times n$ quasimatrices and $C, D$ are discrete matrices of the size $d \times n$. We first compute the two QR decompositions 
\[
\underbrace{
\begin{bmatrix}
\quasi{A} & \quasi{B}
\end{bmatrix}}_{\infty \times 2n}
= \underbrace{\quasi{Q}}_{\infty \times 2n} \underbrace{R}_{2n \times 2n}, \qquad 
\underbrace{\begin{bmatrix}
C & D
\end{bmatrix}}_{d \times 2n} = \underbrace{\tilde Q}_{d \times d} \underbrace{\tilde R}_{d \times 2n},
\]
and then compute the matrix thin SVD
\[
\underbrace{\begin{bmatrix}
  R\\ \tilde R
\end{bmatrix}}_{(2n + d) \times 2n} = \underbrace{U_R}_{(2n + d) \times 2n} \underbrace{\Sigma}_{2n \times 2n} \underbrace{V^\ast}_{2n \times 2n}.
\]
Thus we obtain the following SVD:
\begin{equation}
  \label{eq:svdquasi}  
\begin{bmatrix}
  \quasi{A} & \quasi{B}\\
  C & D
\end{bmatrix} = 
\begin{bmatrix}
  \quasi{Q}& \quasi{0}_{\infty \times d} \\ 0_{d \times 2n} & \tilde Q
\end{bmatrix} 
\begin{bmatrix}
  R\\ \tilde R
\end{bmatrix} 
=\underbrace{\left(\begin{bmatrix}
  \quasi{Q}& \quasi{0} \\ 0 & \tilde Q
\end{bmatrix} 
U_R
\right)}_{(\infty + d) \times 2n}
\Sigma V^*.
\end{equation}

As we can see, these computations crucially rely on the QR factorization of quasimatrices, which is conveniently available in Chebfun. Computing the SVD of these objects then becomes relatively straightforward and efficient.

\section{Least-squares rectangular eigenvalue solver}\label{sec:LSeig}
Having discussed how to compute the SVD of quasimatrix-matrix objects, we are now in a position to introduce our algorithm LSeig for ODE eigenvalue problems, based on solving~\eqref{eq:goal2eig}. 
As discussed above, the key is to solve a quasimatrix rectangular eigenvalue problem. We start by reviewing the Ito-Murota algorithm for the discrete case.

\subsection{Discrete case: Ito-Murota algorithm}\label{sec:IM}
Consider an $m\times n$ ($m>n$) rectangular eigenvalue problem 
\[
A {\bf x} = \lambda B {\bf x}.
\]
As noted previously, there is usually no exact solution. Furthermore, the number of local minima for $h(\sigma) = \sigma_{\min}(A-\lambda B)$ is unknown. To tackle these difficulties, Ito and Murota adopt a perturbation approach via total least-squares:
minimize $\|[\Delta A\ \Delta B]\|_F$ such that $n$ exact solutions exist for the perturbed problem: 
\begin{equation}
  \label{eq:impert}
(A+\Delta A)X = (B+\Delta B)X
\Bigg[\begin{smallmatrix}
  \lambda_1 &&\\&\ddots &\\ && \lambda_n  
\end{smallmatrix}\Bigg],  
\end{equation}
and the vectors $x_i$ are linearly independent. They show that this can be solved as follows: Take the SVD
\[
\begin{bmatrix}
A & B
\end{bmatrix}
= U\Sigma V^* = 
\begin{bmatrix}
U_1 & U_2
\end{bmatrix}
\begin{bmatrix}
\Sigma_{1:n}&\\& \Sigma_{n+1:2n}  
\end{bmatrix}
\begin{bmatrix}
V_{11}^*& V_{21}^*\\
V_{12}^*& V_{22}^*
\end{bmatrix}, 
\]
then solve the square generalized eigenvalue problem 
$ V_{11}^*X = V_{21}^*X\Lambda$ by the standard QZ algorithm. 
With these $X$ and $\Lambda=\mbox{diag}(\lambda_1,\ldots,\lambda_n)$, we have~\eqref{eq:impert} 
with  $\|[\Delta A\ \Delta B]\|_F = \sqrt{\sum_{i=1}^n(\sigma_{n+i}([A\ B]))^2}$.
 As explained by Ito and Murota~\cite{itomurota2016}, their  algorithm is equivalent to solving $(U_1^*A)X=(U_1^*B)X\Lambda$, that is, it is a projection onto the subspace spanned by the $n$ leading left singular vectors of the two matrices $[A\ \ B]$, and hence the residuals $A {\bf x}-\lambda B {\bf x}$ of the outputs are orthogonal to $U_1$.

\subsection{Continuous case}\label{cont_case:sec}
We rewrite $\quasi{A} =
\begin{bmatrix}
\mathcal{L}\quasi{U}\\B        
\end{bmatrix}
$ and $\quasi{B}=\begin{bmatrix}
\quasi{U}  \\
0
\end{bmatrix}$ in~\eqref{eq:goal2eig} and consider the quasimatrix-matrix rectangular generalized eigenvalue problem 
\[
\quasi{A} {\bf x} = \lambda \quasi{B} {\bf x}, \quad \quasi{A},\quasi{B}\in\mathbb{C}^{(\infty+d)\times n}. 
\]
As in Ito-Murota, we attempt to 
solve for the minimal $\|[\Delta \quasi{A}\ \Delta \quasi{B}]\|_F$ such that 
\begin{equation}  \label{eq:IMquasi}
(\quasi{A}+\Delta \quasi{A})X = (\quasi{B}+\Delta \quasi{B})X
\Bigg[\begin{smallmatrix}
  \lambda_1 &&\\&\ddots &\\ && \lambda_n  
\end{smallmatrix}\Bigg].   
\end{equation}
The direct analogue of Ito-Murota is to use the SVD (for $[\quasi{A}\ \quasi{B}]=\big[
\begin{smallmatrix}
\mbox{quasimatrix} & \mbox{quasimatrix}\\
\mbox{matrix} & \mbox{matrix}
\end{smallmatrix}
\big]$, as described in~\eqnref{eq:svdquasi})
\[[\quasi{A}\ \quasi{B}]=\quasi{U}\Sigma V^* = 
[\quasi{U}_1\ \quasi{U}_2]
\begin{bmatrix}
\Sigma_{1:n}&\\& \Sigma_{n+1:2n}  
\end{bmatrix}
\begin{bmatrix}
V_{11}^*& V_{21}^*\\
V_{12}^*& V_{22}^*
\end{bmatrix}, 
\]
where $\quasi{U}_1$ is a quasimatrix-matrix of size $(\infty + d) \times n$, and then reduce the problem to  the fully discrete, $n\times n$ generalized eigenvalue problem
\[
(\quasi{U}_1^*\quasi{A})X=(\quasi{U}_1^*\quasi{B})X\Lambda. 
\]
In the supplementary materials we show that this actually minimizes $\|\Delta \quasi{A}\ \Delta \quasi{B}\|_F$. 

This completes the algorithm for solving~\eqref{eq:diffeq}. An extension to the generalized problem~\eqref{eq:diffeq:gep} is straightforward: all we need is to allow for the second operator $\mathcal{L}_B$ to act on the basis. 

Finally, as mentioned in the introduction, our approach can be adapted easily to boundary conditions that depend affinely on $\lambda$. To do this we 
represent the boundary conditions as follows: 
find  $B_A,B_B\in\mathbb{C}^{d\times n}$ such that $b_i(\quasi{U}{\bf c},(\quasi{U} {\bf c})',\ldots,(\quasi{U} {\bf c})^{(d)},\lambda) =0$ is equivalent to $(B_A-\lambda B_B)[c_1,\ldots,c_n]^T=0$, where $u=\sum c_i u_i$. Putting these together, for the general problem~\eqref{eq:diffeq:gep} with $\lambda$-dependent boundary conditions 
 we solve instead of \eqref{eq:goal2eig} 
\[
\minimize_{{\bf c}\in \mathbb{C}^n}\left\|
  \begin{matrix}
    \begin{bmatrix}
\mathcal{L}_A \quasi{U}\\
B_A
    \end{bmatrix} {\bf c}
-\lambda 
\begin{bmatrix}
\mathcal{L}_B \quasi{U}  \\
B_B
\end{bmatrix} {\bf c}
  \end{matrix}
\right\|_2. 
\]

We now summarize the algorithm. 
\begin{algorithm}[h!]
\begin{algorithmic}[1]
\STATE Form quasimatrices $\quasi{A} = \mathcal{L}_A \quasi{U}, \quasi{B} = \mathcal{L}_B \quasi{U}$. 
\STATE Represent boundary conditions as $(B_A-\lambda B_B) {\bf c} = 0$, where $u = \quasi{U} {\bf c}$. 
\STATE Compute SVD of $(\infty+d)\times 2n$ quasimatrix-matrix (as in~\eqref{eq:svdquasi})
\begin{equation}
  \label{eq:svdalg}
  \begin{bmatrix}
    \quasi{A}& \quasi{B}\\
B_A& B_B
  \end{bmatrix} = 
[\quasi{U}_1\ \quasi{U}_2]
\begin{bmatrix}
  \Sigma_1 &\\
& \Sigma_2
\end{bmatrix}V.
\end{equation}
\STATE Solve the $n\times n$ discrete eigenvalue problem
\begin{equation}
  \label{eq:lseigeig}
\left(\quasi{U}_1^*
\begin{bmatrix}
  \quasi{A}\\B_A
\end{bmatrix}\right)X = 
\left(\quasi{U}_1^*
\begin{bmatrix}
  \quasi{B}\\B_B
\end{bmatrix}\right)X\Lambda.  
\end{equation}
\STATE Output $(\lambda_i,\quasi{U} {\bf c}_i)$ for $i=1,\ldots,n$ (or its subseet) as eigenpairs, 
where $X=[{\bf c}_1,\ldots, {\bf c}_n]\in\mathbb{C}^{n\times n}$ and $\Lambda = \mbox{diag}(\lambda_i)$.
\end{algorithmic}
\caption{LSeig: Least-squares method for  $\mathcal{L}_A u = \lambda \mathcal{L}_B u$, and boundary conditions $b_i(\quasi{U} {\bf c},(\quasi{U} {\bf c})',\ldots,(\quasi{U} {\bf c})^{(d)}) = 0$.
$\quasi{U}:=[u_1,\ldots,u_n]$ are the user-defined basis functions. 
} \label{LSeig:alg}
\end{algorithm}

Remarks:
\begin{itemize}
\item The eigenvalue problem \eqref{eq:lseigeig} gives $n$ eigenpairs, and hence LSeig outputs $n$ solutions $(\lambda_i,\quasi{U}c_i)$ for $i=1,\ldots,n$. Among these, we are often interested in a specific (e.g. largest) eigenpair. In any case, since usually only some of the computed eigenpairs are accurate, we recommend examining the residuals $\|\mathcal{L}_A\quasi{U}c_i  - \lambda_i \mathcal{L}_B \quasi{U}c_i\|_2$ and accepting those that are small enough; see~\eqref{relres_quasi:eq}.
\item No condition is imposed on the basis functions in $\quasi{U}$; they can be chosen arbitrarily. Hence, standard choices such as Chebyshev or Fourier (for periodic problems) are applicable, but in addition, other basis functions are equally applicable. We explore nonstandard choices in our forthcoming experiments. 

\item As with LSode, LSeig has a flavor of both collocation and Galerkin: the residual $\begin{bmatrix}  \mathcal{L}_Au-\lambda\mathcal{L}_Bu  \\ (B_A-\lambda B_B) {\bf c}\end{bmatrix}$ is orthogonal to $\quasi{U}_1$.

\item The dominant computational cost is in computing the SVD and solving the generalized eigenvalue problem. 
Computing the SVD in Chebfun requires roughly $O(kn^2)$ operations, where $k$ is the maximum degree of $u_i$, and the generalized eigenproblem requires  $O(n^3)$ operations, with a total cost of $O((n+k)n^2)$. 

\item We can obtain straightforward variants by introducing weighting to specific rows of~\eqref{eq:svdalg} to give more weight to e.g. the boundary conditions (if enforcing them tightly is important), and scaling the columns (corresponding to diagonal scaling, as employed in the MATLAB {\tt eig} command).

\end{itemize}

\subsection{Pseudospectra of quasimatrices}
Pseudospectra provide insights into what LSeig does. Let $\quasi{A}$ and $\quasi{B}$ be $\infty\times n$ quasimatrices. We define the spectrum of the pair $(\quasi{A}, \quasi{B})$ as follows:
\begin{equation} \label{eigQuasi:eq}
\Lambda(\quasi{A}, \quasi{B}) := \{ \lambda \in \Cz: \exists 0 \neq {\bf v}\in\Cz^{n},\ \forall x \in [-1, 1],\ \quasi{A}(x) {\bf v} = \lambda \quasi{B}(x) {\bf v}  \}.
\end{equation}

We mentioned previously that rectangular matrix pencils usually do not have eigenvalues. Analogously, the spectrum of the quasimatrix pair $(\quasi{A}, \quasi{B})$ is often empty. We nonetheless wish to discuss points that are nearly in the spectrum. 

A nice tool for analyzing the behavior of the spectrum of {\em matrices} and {\em operators}~\cite{davies2000pseudospectra} under small perturbations is the pseudospectrum; see~\cite{trefethenEmbree} and~\cite{Wright02} for a detailed discussion of the various aspects of pseudospectra in those situations. This notion can be extended readily to the case of quasimatrices. Let $\epsilon>0$ be arbitrary and $\sigma_{\min}$ denote the smallest singular value. We call
\begin{equation} \label{pseudoQuasi:eq}
\Lambda_{\epsilon}(\quasi{A}, \quasi{B}) = \bigg\{ z \in \Cz: \frac{\sigma_{\min}(z\quasi{B} - \quasi{A})}{\sqrt{1 + |z|^2}} < \epsilon \bigg\},
\end{equation}
the $\epsilon$-pseudospectrum of the quasimatrix pair $(\quasi{A}, \quasi{B})$, see~\cite{boutry05, das2013solving}. Obviously $\Lambda(\quasi{A}, \quasi{B}) \subseteq \Lambda_{\epsilon}(\quasi{A}, \quasi{B})$, for every $\epsilon>0$. 
More generally, $\Lambda_{\hat\epsilon}(\quasi{A}, \quasi{B}) \subseteq \Lambda_{\epsilon}(\quasi{A}, \quasi{B})$ if $\hat\epsilon<\epsilon$. By minimizing the perturbation so that solutions exist as in~\eqref{eq:IMquasi}, the Ito-Murota algorithm can be seen as a method for finding a set of points in the $\epsilon$-pseudospectrum of $(\quasi{A}, \quasi{B})$, for small values of $\epsilon$.

\begin{example}\label{chebMonde:ex} \normalfont
Let $\quasi{A}(x) = [T_0(x), T_1(x), \cdots , T_5(x)]$ and $\quasi{B}(x) = [P_0(x), P_1(x),\cdots,\allowbreak P_5(x)]$ defined on $[-1,1]$, where $P_i$ denotes the $i$-th Legendre polynomial\footnote{In Chebfun \cite{Chebfun14} these $\infty \times 6$ quasimatrices can be readily constructed with the commands {\tt A = chebpoly(0:5); B = legpoly(0:5);}}. It can be verified that $\Lambda(\quasi{A}, \quasi{B}) = \{1, 1, \frac{4}{3}, \frac{8}{5}, \frac{64}{35}, \frac{128}{63}\}$ and that for example ${\bf v}_3 = [\frac{-1}{\sqrt{2}}, 0, \frac{1}{\sqrt{2}}, 0, 0, 0]^T$ is the eigenvector corresponding to $\lambda_3 = \frac{4}{3}$. Figure~\ref{fig:basicDef2} illustrates $\Lambda_{\epsilon} (\quasi{A}, \quasi{B})$ on a grid of $z$-values. 
\begin{figure}[!h]
\centering
\includegraphics[width=80mm]{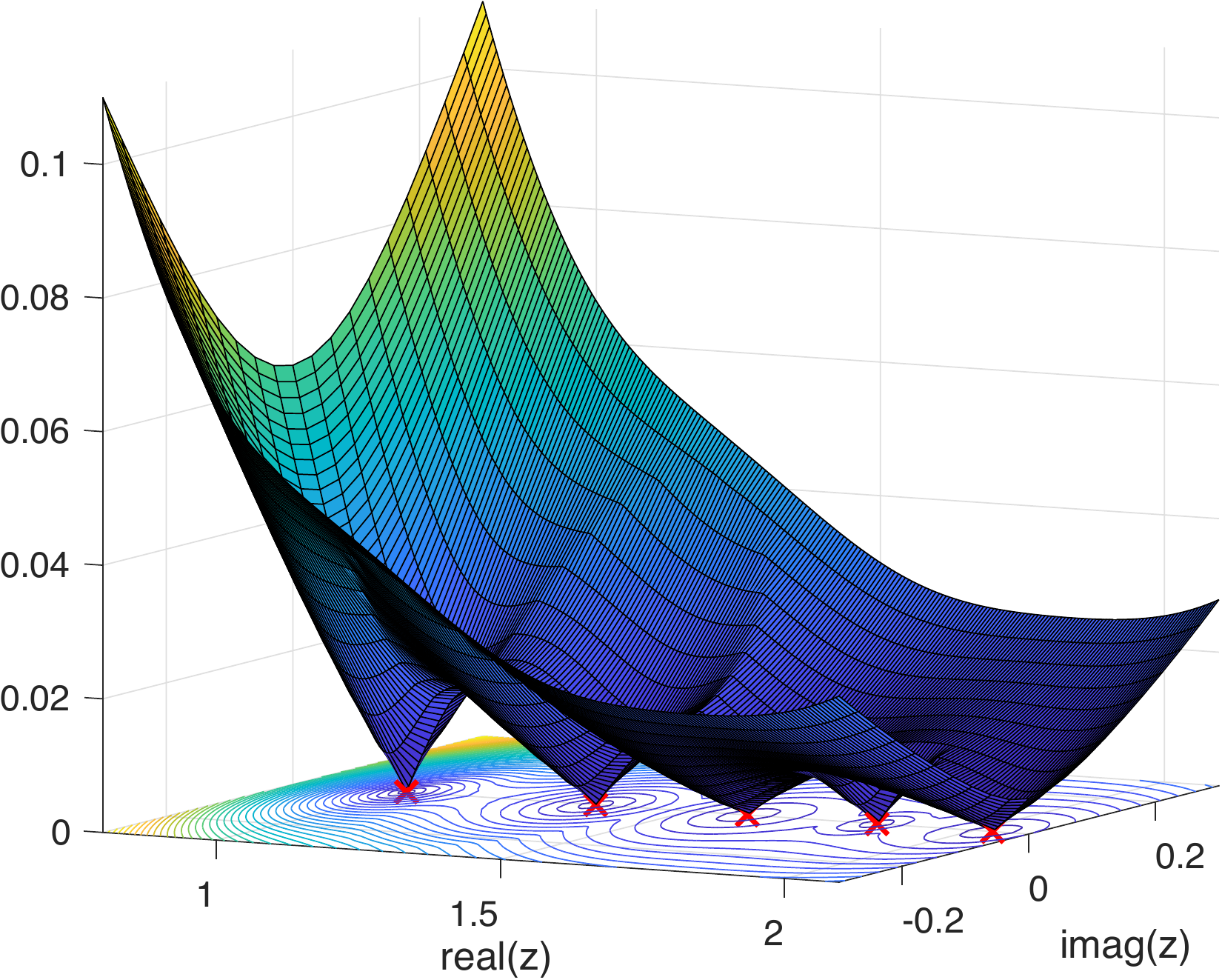}
\caption{Graphical inspection of the spectrum and the pseudospectra of the $\infty \times 6$ quasimatrix pencil of Example~\ref{chebMonde:ex}. Red crosses denote the eigenvalues.} \label{fig:basicDef2}
\end{figure}

\end{example}

\section{Boundary conditions}\label{sec:bc}
Here we discuss various aspects of boundary conditions. We discuss a variant of our algorithm where the boundary conditions are imposed exactly (instead of being minimized together with the residual).  We then discuss boundary conditions that arise specifically in eigenvalue problems, namely those that depend on $\lambda$. 

\subsection{Imposing boundary conditions exactly}\label{sec:bcexact}
We can modify Algorithm~\ref{LSeig:alg} to impose boundary conditions exactly by restricting the projection space to contain the parts representing the boundary conditions. 
Conceptually,  this is done by introducing an infinite weight in the final $d$ rows of~\eqref{eq:svdalg}. This forces the leading $d$ left singular vectors to span $
\begin{bmatrix}
  0_{\infty\times d}\\ I_d
\end{bmatrix}$. We then choose the remaining ($n-d$)-dimensional part by the SVD of the remaining part, that is, the quasimatrix $[\quasi{A}\ \quasi{B}]$. 

Namely, steps 3,4 in Algorithm~\ref{LSeig:alg} become
\begin{itemize}
\item Compute the SVD of the $\infty \times 2n$ quasimatrix
\[
  \begin{bmatrix}
    \quasi{A}& \quasi{B}
  \end{bmatrix} = 
[\quasi{U}_1\ \quasi{U}_2]
\begin{bmatrix}
  \Sigma_1 &\\
& \Sigma_2
\end{bmatrix}V
,\qquad \quasi{U}_1\in\mathbb{C}^{\infty\times (n-d)}.\]
\item Set $\quasi{U}:=
  \begin{bmatrix}
\quasi{U}_1& 0_{\infty\times d}    \\
0_{d\times (n-d)}& I_d 
  \end{bmatrix}$
and solve the $n\times n$ square  eigenvalue problem
\[
\left(\quasi{U}^*
\begin{bmatrix}
  \quasi{A}\\B_A
\end{bmatrix}\right)X = 
\left(\quasi{U}^*
\begin{bmatrix}
  \quasi{B}\\B_B
\end{bmatrix}\right)X\Lambda. 
\]
\end{itemize}
We refer to this algorithm as LSeig-bc.

For ODEs, we can similarly modify Algorithm~\ref{LSode:alg} as follows: 
take the leading $n-d$ left singular functions of $\mathcal{L} \quasi{U}$, call it $\quasi{U}_1$, define $\quasi{Q} =
\begin{bmatrix}
  \quasi{U}_1&0  \\ 0&I_d
\end{bmatrix}$, and find $c$ by imposing $\quasi{Q}^*\left(
\begin{bmatrix}
\mathcal{L}U\\ B   
\end{bmatrix} {\bf c}
-
\begin{bmatrix}
f\\ {\bf f_b}
\end{bmatrix}\right)=0$.

\subsection{$\mathbf{\lambda}$-dependent boundary conditions}\label{sec:lambc}
ODE eigenvalue problems involving boundary conditions that depend on the 
unknown spectral parameter $\lambda$ are an interesting class of non-standard eigenvalue problems.\ Such problems arise, in particular, in elasticity and hydrodynamics \cite{Gheorghiu05, Marletta03}. One example is the Orr--Sommerfeld equation for a liquid film flowing over an inclined plane, with a surface tension gradient which involves boundary conditions that depend linearly on $\lambda$; see \cite{Gheorghiu96, Greenberg01}. Other problems with $\lambda$-dependent boundary conditions can be found in~\cite{Marletta03}. 

In addition, the development of a Sturm theory (ordering of eigenvalues, studying oscillation of eigenfunctions, etc.) for eigenvalue problems of second-order ODEs with $\lambda$-dependent boundary conditions has been studied in \cite{Hochstadt67, binding97, Chan12}. 

Among the earliest references on nonclassical boundary conditions is Birkhoff~\cite{birkhoff08issue4, birkhoff08issue2}.\ Tamarkin \cite{tamarkin17,tamarkin28} considered problems where the differential equation and the boundary conditions depend polynomially on $\lambda$. For more historical notes on $\lambda$-dependent boundary conditions, see \cite{Tretter01}. 

From the algorithmic point of view it is known that imposing boundary conditions that depend on $\lambda$ is challenging for spectral collocation methods and the tau method is considered the method of choice; see \cite[p. 112]{Gheorghiu05}. As seen above, it is straightforward for our least-squares framework to deal with boundary conditions that depend affinely\footnote{It is however not straightforward with boundary conditions that depend \emph{polynomially} on $\lambda$. We leave this for future work. } on $\lambda$. Other techniques include the regularized sampling method with which we compare the performance of our framework in Example~\ref{Chanane:ex}.

\section{Integral equation reformulation}\label{sec:integral}
The differentiation operator is unbounded, while integration is compact.\ This roughly implies that differentiation is ill-conditioned whereas integration is a well conditioned operation.\ Thus, it is often a good idea to reformulate differential equations as integral equations, whenever possible~\cite{greengard1991spectral}, \cite{driscoll2010automatic}.

 Here let us consider the advection-diffusion eigenvalue problem on the interval $[a,b]$ 
\[
u''+u'=\lambda u,\quad u(a)=u(b)=0. 
\]
In the integral reformulation we take $v(x) := u''(x)$ as the unknown. Then, $u'(x) =\alpha + \int_a^x v(\eta)\ d\eta$, $u(x) = \alpha x+\beta + \int_a^x \int_a^t v(\eta)\ d\eta\ dt$ and the problem becomes 
\[
v(x) + \alpha + \int_a^x v(\eta)\ d\eta=
\lambda \left( \alpha x+\beta + \int_a^x\int_a^t v(\eta)\ d\eta\ dt \right),
\]
with boundary conditions 
\begin{equation}  \label{eq:bcint}
\alpha a+\beta=0,\quad 
 \alpha b+\beta + \int_a^b\int_a^t v(\eta)\ d\eta\ dt=0.   
\end{equation}
We can formulate these in a least-squares eigenvalue problem. 
To do so we write $v(x)=\sum_{i=0}^{n-1} c_i T_i(x)$, and on the left-hand side 
we take the basis to be $[T_0,\dots, T_{n-1},1,\allowbreak 0]$ (note that this quasimatrix is rank-deficient). 
The right-hand side has basis $[T_0, \dots,\allowbreak T_{n-1}, x, 1]$ and the eigenvector is $[c_0,c_1\dots,c_{n-1},\alpha,\beta]$. We used monomials to represent constant terms resulting from each integration but, as in Example~\ref{OS:ex}, the Chebyshev basis could also be used. To summarize, LSeig takes the two input quasimatrices $\quasi{A} = [\mathcal{I}_A \quasi{Q},\ 1,\ 0]$ and $\quasi{B} = [\mathcal{I}_B\quasi{Q},\ x,\ 1]$ where $\mathcal{I}_A v(x) := v(x) +  \int_{a}^x v(\eta)\ d\eta$ and $\mathcal{I}_B v(x) :=  \int_a^x \int_a^t v(\eta)\ d\eta\ dt$, together with the boundary conditions~\eqref{eq:bcint}. The rest is the same as in Algorithm~\ref{LSeig:alg}. We illustrate this idea for the Orr-Sommerfeld problem in Example~\ref{OS:ex}.

\section{Experiments}\label{sec:exp}
This section reports the performance of our least-squares framework for various problems, focusing on the flexibility of our approach and the accuracy of the computed eigenpairs. All numerical experiments were done using Chebfun on a standard laptop. Inputs to our LSeig and LSeig-bc algorithms include quasimatrices $\quasi{A}$ and $\quasi{B}$ together with matrices representing boundary conditions. Another input is the tolerance \texttt{tol} on the following relative residual ensuring that we only output eigenpairs satisfying
\begin{equation}
\label{relres_quasi:eq}
\frac{\|\quasi{A} V - \quasi{B} V \Lambda \|_2}{\|\quasi{A} V\|_2} < \mbox{ \texttt{tol}.}
\end{equation}
We report \texttt{tol} in each experiment. When we report residuals with LSeig we report the augmented residual that includes the residual of the boundary conditions. In most of our experiments, LSeig-bc tended to give slightly more accurate results.

\begin{example}\label{inviter}
\normalfont
\emph{Tailor-made basis functions.} 
The flexibility of our approach is advantageous especially when the solutions are complicated (e.g.\ nonsmooth), so that classical approaches based on global polynomials are inefficient.\ When the solution is nonsmooth, for example oscillatory or has singularities (e.g. of absolute-value type), it is much desirable to incorporate tailor-made basis functions. 

To illustrate this, we consider the Airy eigenvalue problem 
\begin{equation}  \label{eq:airyeig}
\mathcal{L}u: = \epsilon u'' + xu = \lambda u, \qquad 
u(\pm 1) = 0.  
\end{equation}
Here we target the smallest eigenpair. 
When $|\epsilon|\ll 1$, the eigenfunctions are highly oscillatory (see right panel of Figure~\ref{fig:inviter}), and representing them accurately with polynomials would require a very high degree, making the computation inefficient. 

To overcome this difficulty, one can consider using subspace methods that successively build a subspace that is designed to be rich in the desired eigenspace. Examples for matrix eigenvalue problems 
include Krylov methods and their variants (shift-and-invert, rational) the Rayleigh quotient iteration, Jacobi-Davidson~\cite{jdoriginal}  and generalized Davidson method, each differing in how the subspace is computed.

Here we illustrate a subspace built via the inverse iteration: given an approximate eigenfunction $\tilde u$, we solve the ODE  
\begin{equation}  \label{eq:airylinsys}
\mathcal{L}u = \tilde u 
\end{equation}
for $u$, and update $u:=\frac{\tilde u}{\|\tilde u\|_2}$, and repeat until convergence. 
To solve the ODE~\eqref{eq:airylinsys} we used the Olver-Townsend algorithm~\cite{olver2013fast}, which is also available in Chebfun\footnote{Using the flag \verb+options.discretization = 'coeffs'+.}. 
We compare inverse iteration against LSeig using the subspace generated by the iterates $\tilde u$, that is, 
$\mbox{span}(\quasi{U})=\mbox{span}(u_0,\mathcal{L}^{-1}u_0,\mathcal{L}^{-2}u_0,\ldots,\mathcal{L}^{-(n-1)}u_0)$, the Krylov subspace with respect to the inverse $\mathcal{L}^{-1}$ as in Section~\ref{sec:LSODEex}. Here $u_0$ is the initial approximation, which we take to be $u_0(x)=(x-1)(x+1)$; unlike in Section~\ref{sec:LSODEex}, we start the subspace from $u_0$ as it satisfies the boundary conditions.

The results for $\epsilon=10^{-8}$ are shown in Figure~\ref{fig:inviter}. 
We see that (i) both inverse iteration and LSeig together with the inverse iteration Krylov subspace are effective algorithms for the problem (both are dramatically better than global polynomials, which would need $n\gg 10^3$), and (ii) LSeig outperforms inverse iteration by finding ``the best'' solution in the subspace; this difference is analogous to that between the standard power method and Lanczos method for symmetric eigenvalue problems.

\begin{figure}[htbp]
  \begin{minipage}[t]{0.495\hsize}
\includegraphics[width=0.91\textwidth]{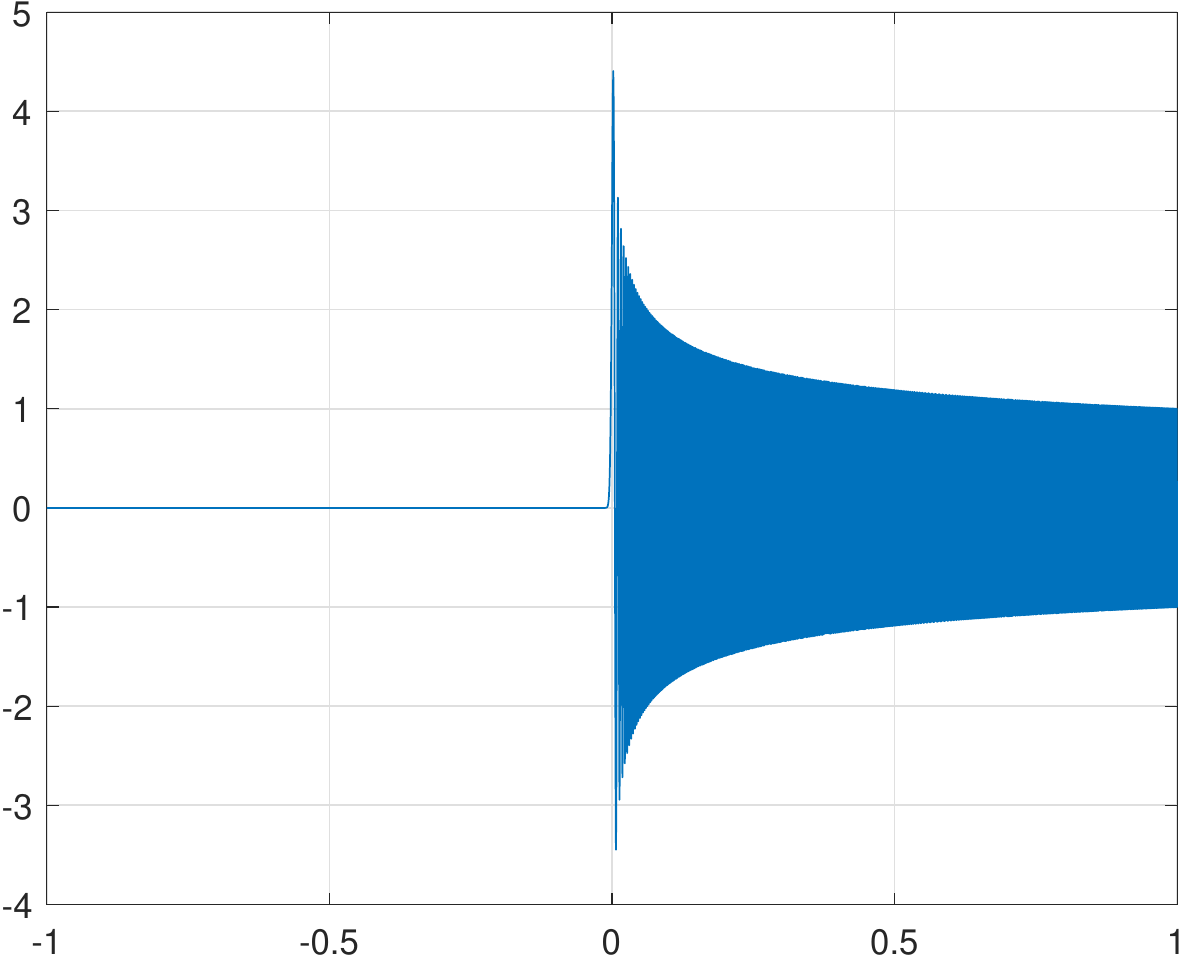}
  \end{minipage}   
  \begin{minipage}[t]{0.495\hsize}
\includegraphics[width=\textwidth]{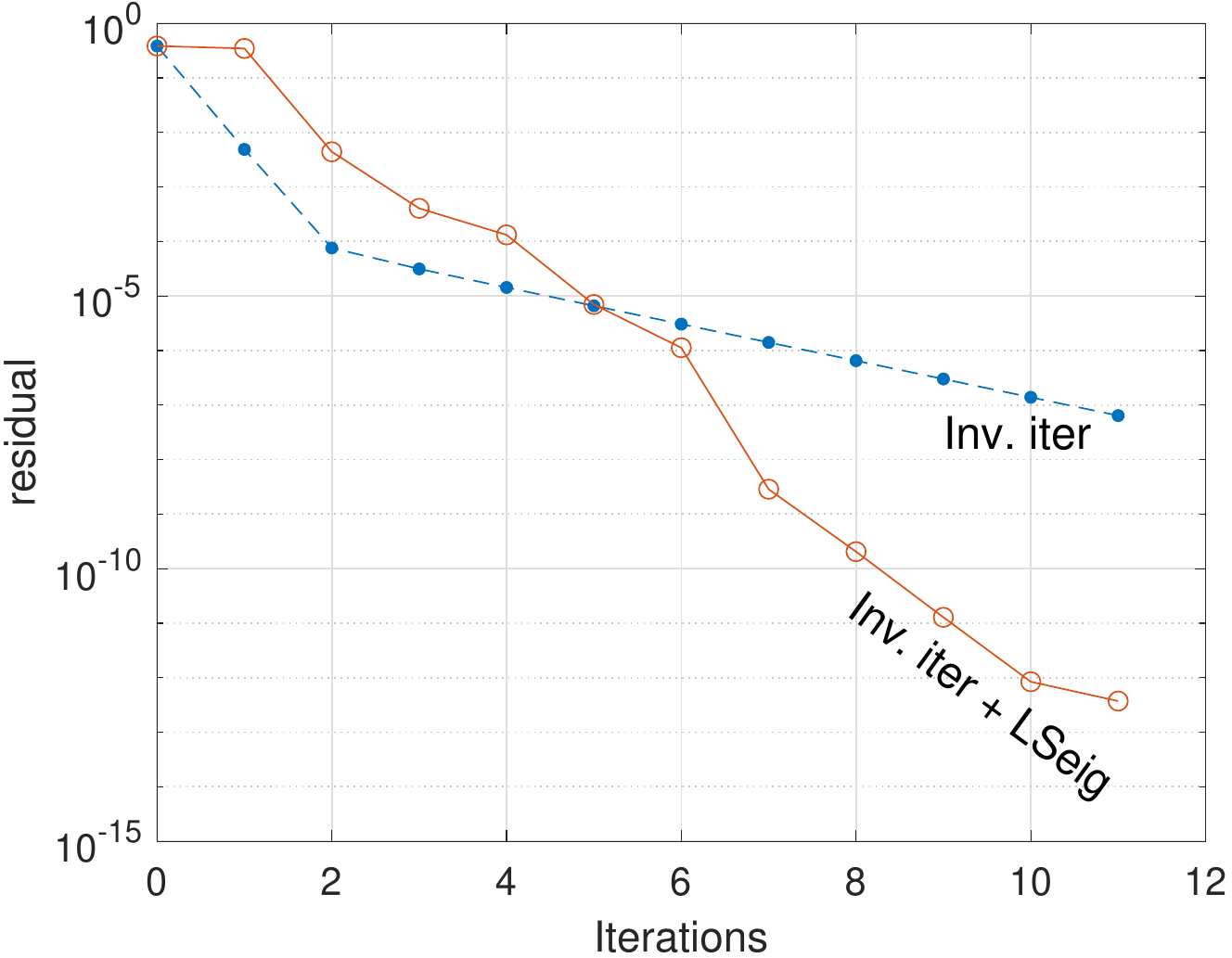}
  \end{minipage}
  \caption{
Solving the Airy eigenvalue problem \eqref{eq:airyeig} with $\epsilon=10^{-8}$, using (i) inverse iteration (Inv. iter)  
and (ii) LSeig using the subspace obtained by inverse iteration. 
Right: Solution $\hat u$. Left: Residual history $\|\mathcal{L}\hat u - \hat\lambda \hat u\|_2$. }
  \label{fig:inviter}
\end{figure}

We note that for $\epsilon<10^{-6}$, the solution becomes too oscillatory for Chebfun's command {\tt eigs} based on a collocation method to converge. 
When $\epsilon<10^{-9}$ we further observed that inverse iteration fails to converge, as the ODE~\eqref{eq:airylinsys} becomes increasingly difficult to solve. Using the same subspace, LSeig nonetheless obtained solutions with residual smaller than $10^{-10}$. LSeig-bc (imposing boundary conditions exactly) led to similar results.  Here and in the next example the \texttt{tol} value in LSeig  was selected to $1$ so that all quantities are initially output, and we chose the eigenpair with the smallest residual.
\end{example}

\begin{example}~\label{abs}
\normalfont
\emph{Problems with singularities.}
We next consider Example 20 in \cite[App.~B]{trefethen2017exploring}, a 1D Schr\"odinger equation 
\begin{equation}  \label{eq:schrodinger}
\mathcal{L}u=-h^2u'' +V(x)u = \lambda u, \qquad 
V(x)=|x|, 
\end{equation}
with $h=0.1$ and boundary conditions $u(\pm 3)=0$ (the domain is $[-3,3]$). Here the potential $V(x)$ has a singularity at $x=0$, and this forces the eigenfunctions to also have a (weaker but genuine) singularity: they are twice differentiable but not more. 

The nonsmoothness of $u$ presents difficulties for global polynomial-based methods, since as is well known, polynomials cannot approximate such functions efficiently (the convergence is $O(1/n)$ or slower, where $n$ is the degree). 

To remedy this, one can employ locally supported, piecewise polynomial basis functions as follows: 
\begin{equation}
  \label{eq:splitu}
u_{2k-1}(x) = 
\begin{cases}
T_k(2(x+\frac{1}{2})), &  x< 0\\
    0,              & x\geq 0
\end{cases}, \qquad 
u_{2k}(x) = 
\begin{cases}
0, &  x< 0\\
T_k(2(x-\frac{1}{2})),              & x\geq 0
\end{cases}  .
\end{equation}
We use LSeig with these basis functions, imposing continuity of the solution and its first and second derivatives at the splitting point $x=0$; this can be done by adding three more boundary conditions in the least-squares formulation. 
We compare the algorithm with LSeig using the standard global polynomials $u_i = T_i(x)$. The results are shown in Figure~\ref{fig:schrodinger}, which illustrates the advantages of using local piecewise polynomials when the location and nature of the singularity is known. 

\begin{figure}[htbp]
  \begin{minipage}[t]{0.495\hsize}
\includegraphics[width=0.92\textwidth]{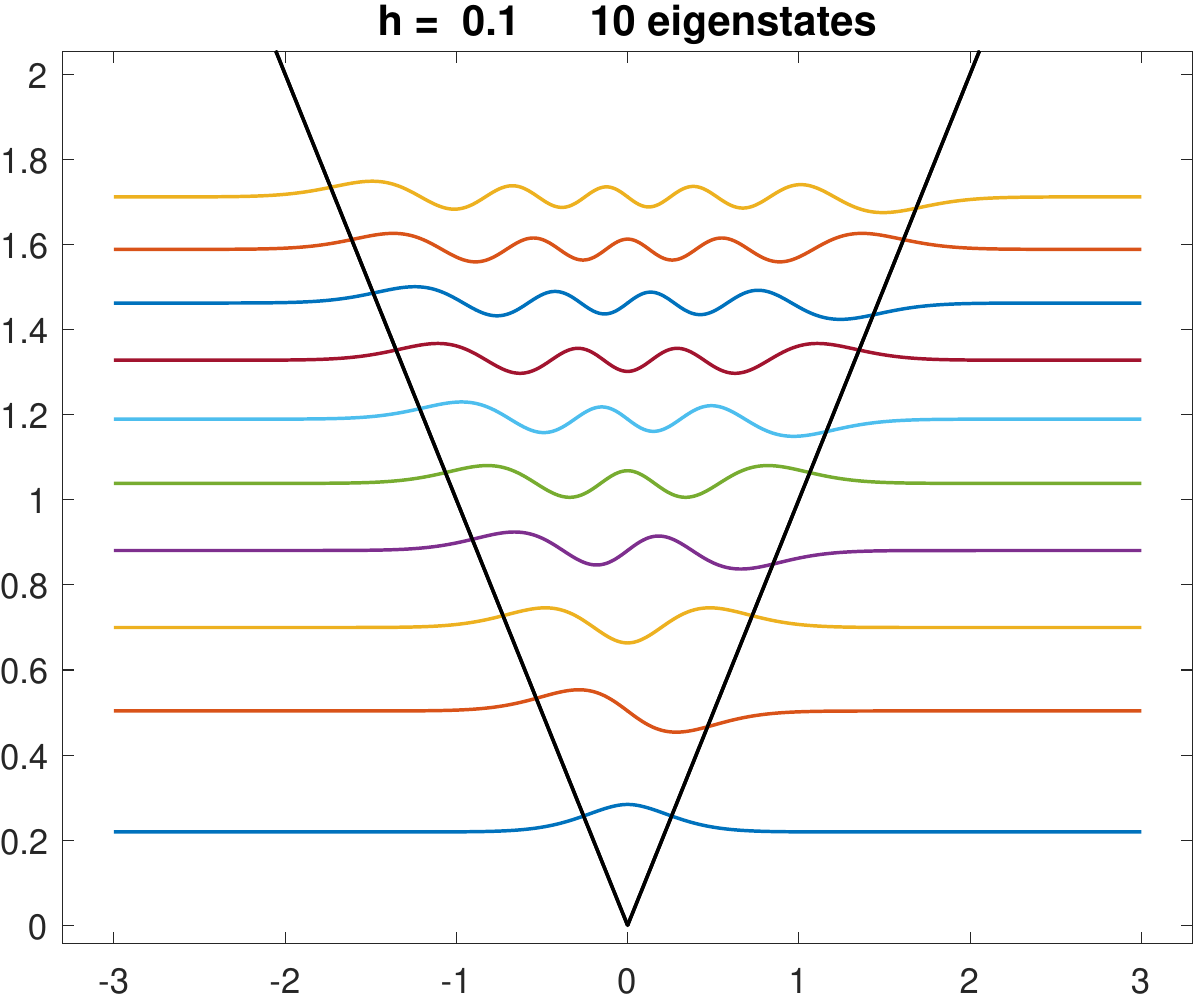}
  \end{minipage}   
  \begin{minipage}[t]{0.495\hsize}
\includegraphics[width=\textwidth]{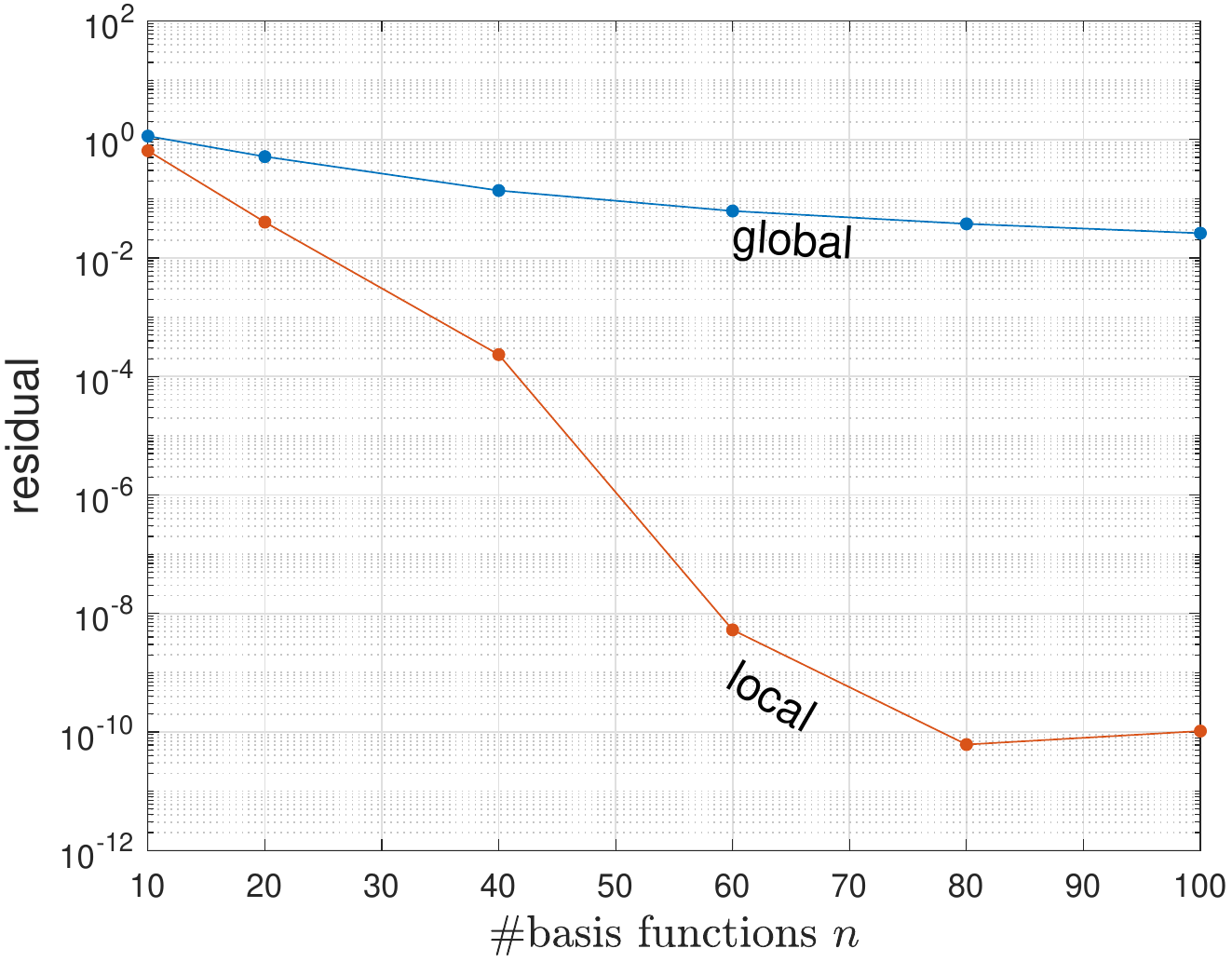}
  \end{minipage}
  \caption{Eigenfunctions (left; reprint of~\cite[App.~B, Ex.~20]{trefethen2017exploring}) and convergence to smallest eigenfunction for~\eqref{eq:schrodinger} using LSeig, with global and local~\eqref{eq:splitu} polynomial basis functions.}
  \label{fig:schrodinger}
\end{figure}
  
\end{example}

It is worth noting that Chebfun's algorithm can also deal with such problems efficiently, by detecting singularities and working with piecewise polynomials~\cite{pachon2010piecewise}.

\begin{example}\label{SL:ex}
\normalfont
\emph{Generalized eigenproblems.} Consider the Sturm-Liouville problem
\[
(e^{3x} u'(x))' + 2e^{3x} u(x) + \lambda e^{3x} u(x) = 0, \qquad u(0) = u(1) = 0.
\]
Starting from 100 Chebyshev polynomials and $\mbox{\texttt{tol}} =10^{-10}$, the LSeig method with exact boundary conditions obtains 41 eigenvalues in $0.14$ seconds. The Chebfun \texttt{eigs} command takes $0.74$ seconds when asked to compute the same number of eigenvalues\footnote{We are not claiming LSeig is a faster algorithm---Chebfun's method finds an appropriate degree adaptively, whereas in LSeig the degree is an input. The point here is that the speed of LSeig is comparable.}. Relative residuals~\eqnref{relres_quasi:eq} are depicted in Figure~\ref{fig:SLres}. We also report the average number of correct digits $d$ by looking at the logarithm of the geometric mean of the relative residuals. Note that the eigenfunctions of this Sturm-Liouville equation should be orthogonal with respect to the weight $e^{3x}$; see~\cite[ch. 13]{Trench21} for instance. The departures from orthonormality in the computed (normalized) eigenfunctions are similar: $2.1\times 10^{-8}$ for LSeig-bc and $1.8 \times 10^{-8}$ for Chebfun \texttt{eigs}. 

\begin{figure}[htbp]
\centering
\includegraphics[width=60mm]{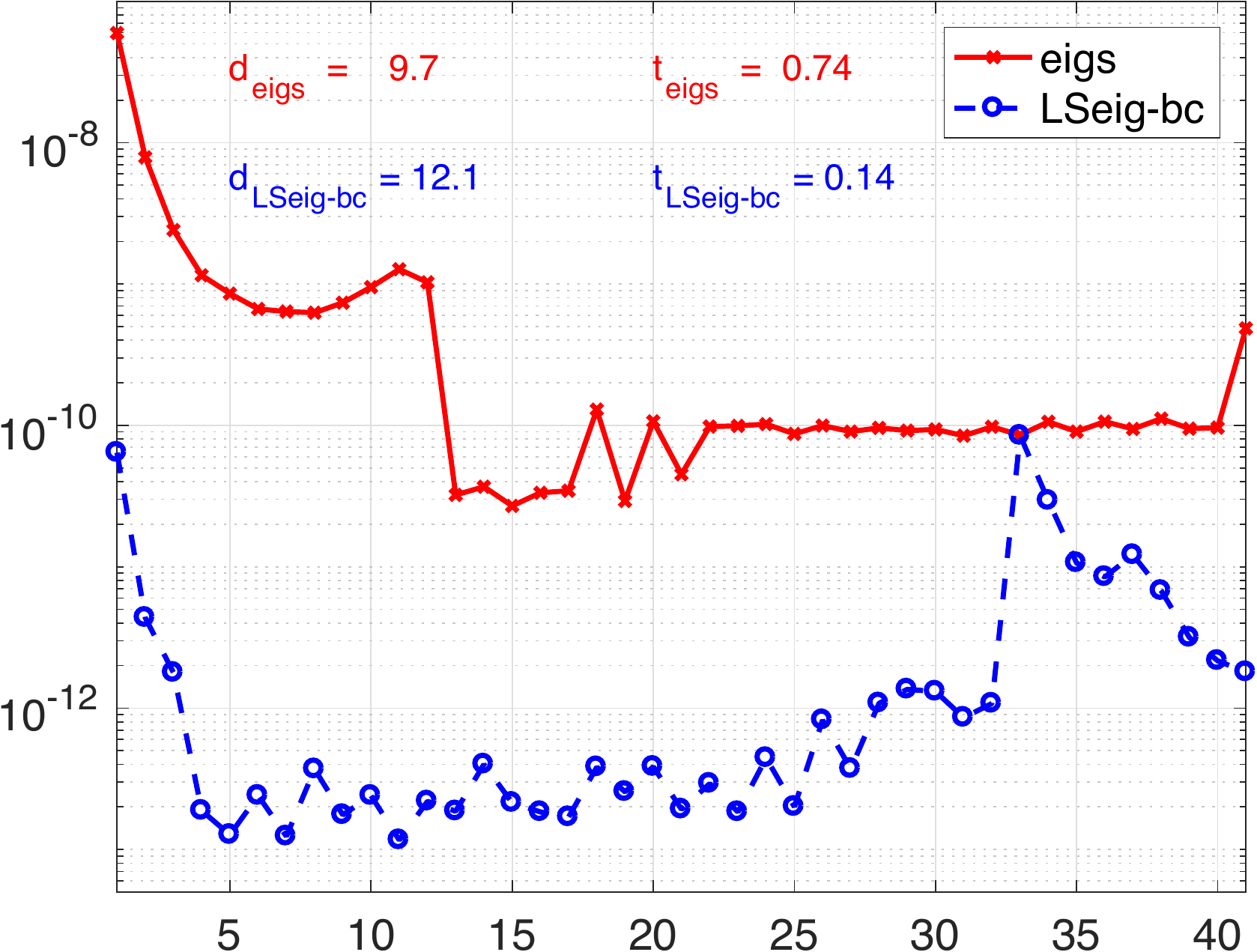}
\caption{Relative residuals~\eqnref{relres_quasi:eq} in 41 computed eigenvalues of the Sturm-Liouville problem of Example~\ref{SL:ex}.}
\label{fig:SLres}
\end{figure}

\end{example}

\begin{example}~\label{OS:ex}
\normalfont
{\em Basis recombination and integral reformulation.}
Let us consider the famous Orr-Sommerfeld equation
\begin{equation}
 \label{eq:OS}
\frac{1}{R}(u''''-2u''+u) - 2iu - i(1-x^2)(u''-u)
= \lambda(u''-u),
\end{equation}
with boundary conditions $u(\pm1) = u'(\pm1) = 0$, which arises in determining conditions for hydrodynamic stability. See \cite{schmid01} for instance. Here, $R$ denotes the Reynolds number, which corresponds roughly to velocity divided by viscosity. We take $R = 5772$, which is about the critical value at which an eigenvalue first crosses into the right half-plane, corresponding to an unstable flow.

We write~\eqnref{eq:OS} as $\mathcal{L}_A u=\lambda\mathcal{L}_B u$ and apply three  methods within our least-squares framework. First of all, we apply LSeig-bc directly to~\eqnref{eq:OS}. The second method applies LSeig with basis recombination using the following basis functions satisfying the basis functions:
\[
(1+x)^2 (1-x)^2 T_i(x), \qquad i = 0,1,\dots, n-1.
\]

We also consider an integral reformulation of~\eqnref{eq:OS} solving for $v:=u''''$, with eigenvector $[c_0,c_1\ldots,c_{n-1},a_0,a_1,a_2,a_3]$, where 
\begin{equation}
\label{solExapnd:eq}
u(x) = \int_{-1}^x \int_{-1}^w \int_{-1}^s \int_{-1}^t v(\eta)\ d\eta \ dw\ ds\ dt+ p(x),
\end{equation}
$v(x) = \sum_{i=0}^{n-1}a_i T_i(x)$, and $p(x)$ which is a polynomial of degree three resulting from integrations is represented with $\sum_{i=0}^3 a_i T_i(x)$. In addition,
\[
u''(x) = \int_{-1}^x \int_{-1}^w v(\eta)\ d\eta \ dw+ p''(x).
\]
Thus, we can write $v = \quasi{Q} c$ where $\quasi{Q} := [T_0,\dots,T_{n-1}]$ and expand $u$ in terms of the columns of the basis quasimatrix $[\quasi{Q}, T_0, T_1, T_2, T_3]$. We rewrite the Orr-Sommerfeld equation as $\mathcal{I}_A v+\mathcal{L}_A p =\lambda (\mathcal{I}_B v+\mathcal{L}_B p)$, where 
\small
\begin{align*}
\mathcal{I}_A v(x) := &\frac{1}{R} v(x) - \Big(\frac{2}{R} +i(1-x^2) \Big) \int_{-1}^x \int_{-1}^w v(\eta)\ d\eta \ dw \nonumber\\
& + \Big(\frac{1}{R} -2i + i(1-x^2) \Big)\int_{-1}^x \int_{-1}^w \int_{-1}^s \int_{-1}^t v(\eta)\ d\eta \ dw\ ds\ dt
\nonumber,
\end{align*}
\[
\mathcal{I}_B v(x) := \int_{-1}^x \int_{-1}^w v(\eta)\ d\eta \ dw - \int_{-1}^x \int_{-1}^w \int_{-1}^s \int_{-1}^t v(\eta)\ d\eta \ dw\ ds\ dt,\nonumber
\]
\normalsize 
$\mathcal{L}_A p(x) := \sum_{i=0}^3 a_i q_i(x)$ and $\mathcal{L}_B p(x):=\sum_{i=0}^3 a_i \tilde q_i(x)$. Therefore, we apply our LSeig-bc algorithm to the quasimatrices $\quasi{A}:=[\mathcal{I}_A \quasi{Q}, q_0,q_1,q_2,q_3]$ and $\quasi{B} := [\mathcal{I}_B \quasi{Q}, \tilde q_0,\tilde q_1,\tilde q_2,\tilde q_3]$ where $q_i=\mathcal{L}_AT_i$ and $\tilde q_i =\mathcal{L}_B T_i$ for $i=0,1,2,3$. Also, LSeig takes two matrices 
$B_A$ and $B_B$ of size $d \times (n+d)$ where each row represents one of the $d = 4$ boundary conditions. This means that $B_B = 0$ as boundary conditions do not depend on $\lambda$. In addition, having \eqnref{solExapnd:eq} in mind, for instance, the row of $B_A$ corresponding with the boundary condition $u(+1) = 0$ contains factors multiplied by $a_i$ and $c_i$ in the following equation:
\[
\sum_{i=0}^{n-1} a_i \int_{-1}^{1} \int_{-1}^w \int_{-1}^s \int_{-1}^t  T_i(\eta)\ d\eta \ dw\ ds\ dt + \sum_{i=0}^{3} c_i T_i(+1) = 0.
\]
Notice the upper bound $x = 1$ of the outer-most integral.

In this example, we use 100 Chebyshev polynomials and set $\mbox{\texttt{tol}} = 10^{-2}$. Figure~\ref{fig:OS_eigs} (left) depicts eigenvalues computed with our least-squares method applied to the integral reformulation in which boundary conditions are imposed exactly.  The right panel depicts the residuals corresponding with eigenvalues computed with our algorithms applied to the integral reformulation and with basis recombination. We see that LSeig-bc applied to the integral reformulation computes accurate results for a larger number of eigenvalues compared with basis recombination. See also Table~\ref{tab:OS}.

\begin{figure}[t]
  \begin{minipage}[t]{0.495\hsize}
      \includegraphics[width=.97\textwidth]{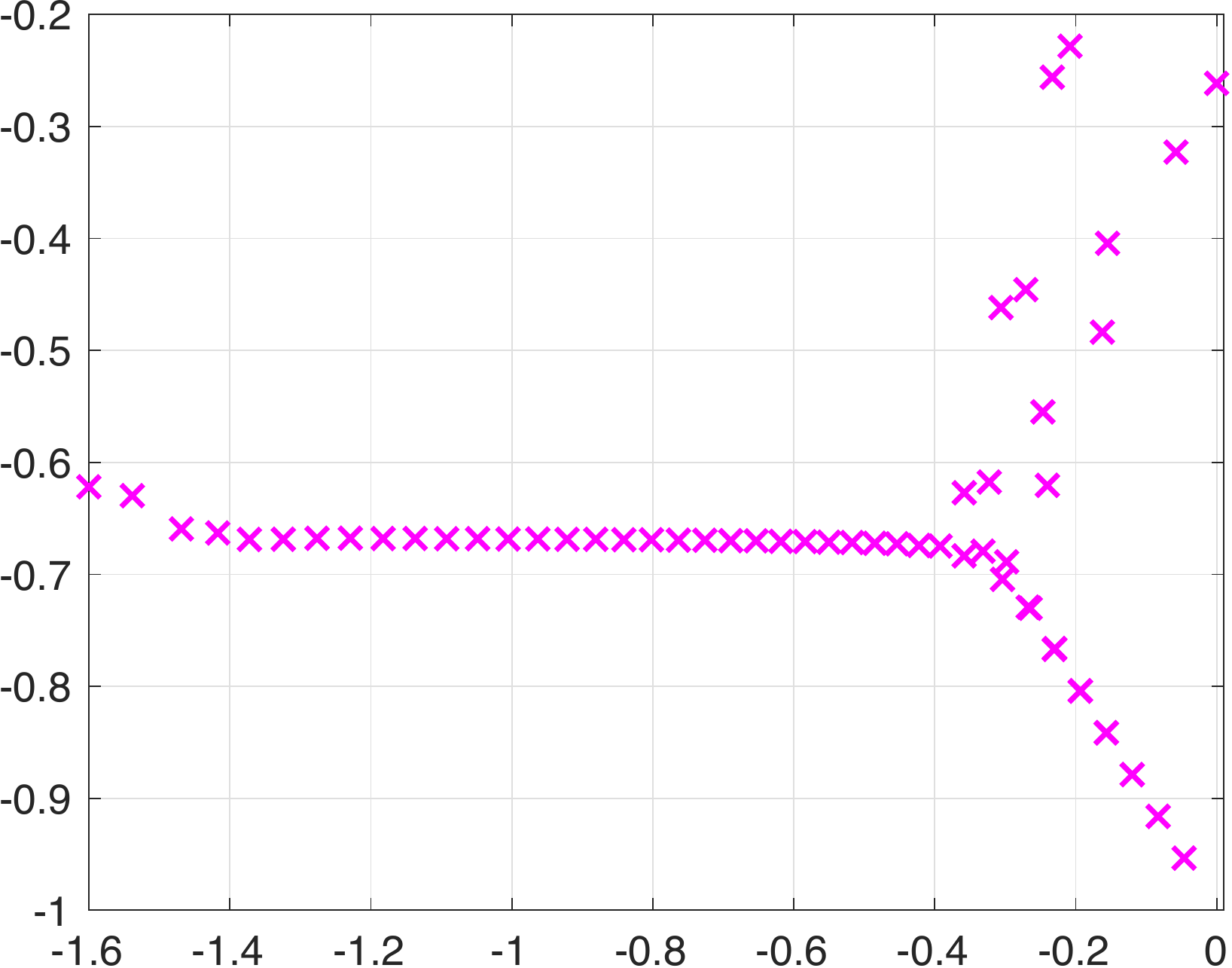}
  \end{minipage}   
  \begin{minipage}[t]{0.495\hsize}
     \includegraphics[width=\textwidth]{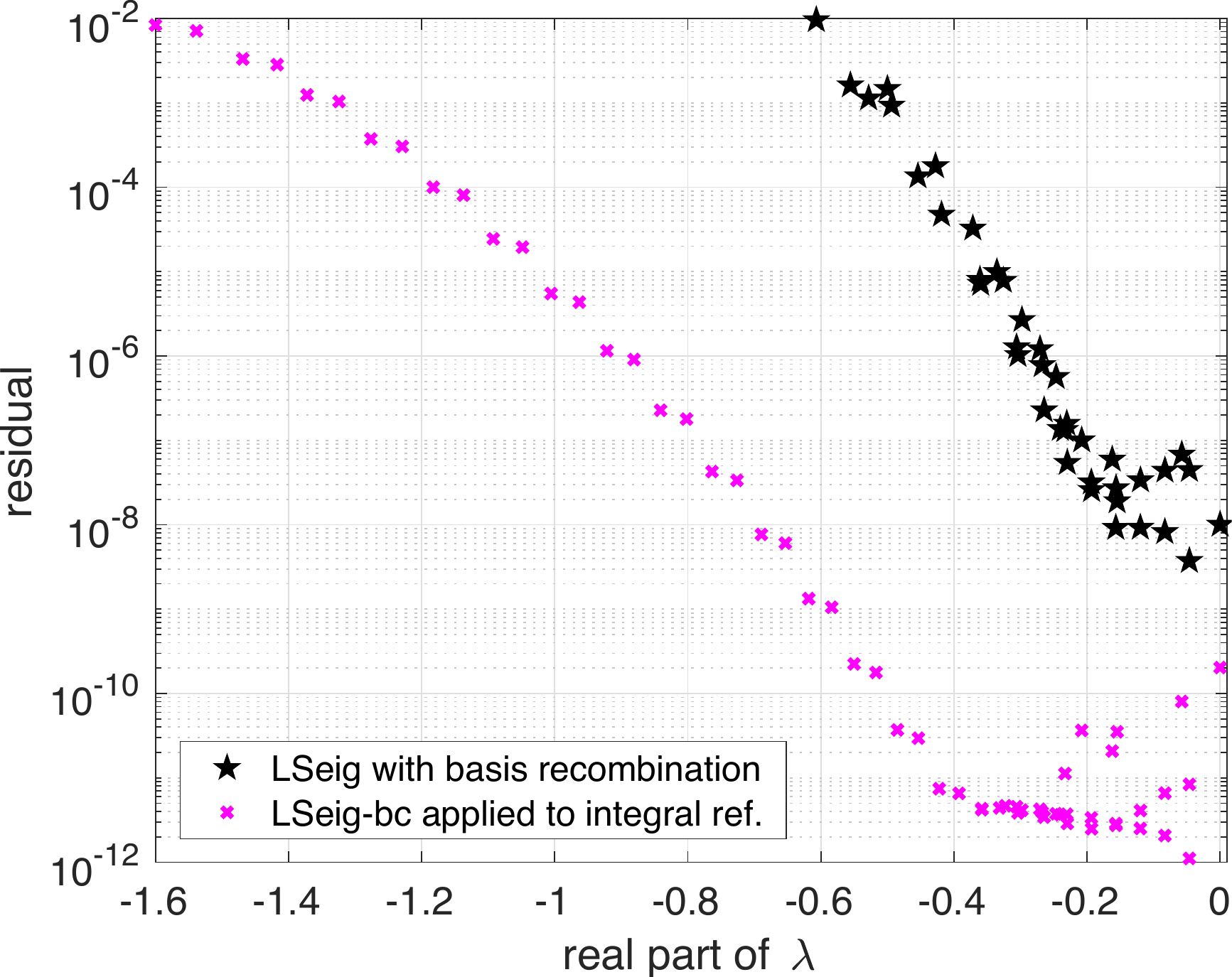}
  \end{minipage}
    \caption{Performance of methods for the Orr-Sommerfeld equation in Example~\ref{OS:ex} Left: 60 eigenvalues computed with integral reformulation. The right-most eigenvalue is $ -7.8191 \times 10^{-5} - 0.26157i$. Right: Relative residuals in~\eqnref{eq:OS} for 60 eigenvalues obtained with integral reformulation and 39 eigenvalues obtained with basis recombination applied to the original differential problem}
    \label{fig:OS_eigs}
\end{figure}

\begin{table}[h]
  \centering
  \caption{Residuals $\|\mathcal{L}_A u - \lambda\mathcal{L}_B u\|_2/\|\mathcal{L}_A u\|_2$ for the six eigenvalues of largest real part in Example~\ref{OS:ex}.}
      \label{tab:OS}
\footnotesize
  \begin{tabular}{c|c|c|c}
\rule{0pt}{2.6ex}  Chebfun eigs &  LSeig-bc  &  LSeig (BR) & LSeig-bc (int)\\
\hline
   $3.2\times 10^{-5}$  & \   $5.9 \times 10^{-5}$  \rule{0pt}{2.6ex} & $1.0 \times 10^{-8}$ & $2.0 \times 10^{-10}$  \\
   $7.0\times 10^{-7} $ &   $1.9\times 10^{-4}$  &  $4.5 \times 10^{-8}$ & $8.4 \times 10^{-12}$  \\
   $3.4\times 10^{-6} $ &   $1.6\times 10^{-4} $ &  $3.7 \times 10^{-9}$ & $1.1 \times 10^{-12}$ \\
   $1.2\times 10^{-5} $ &   $1.1\times 10^{-4}$  &  $6.8 \times 10^{-8}$ & $8.0 \times 10^{-11}$  \\
   $2.8\times 10^{-6} $ &   $3.7\times 10^{-4} $ &  $4.4 \times 10^{-8}$ & $6.6 \times 10^{-12}$ \\
   $9.2\times 10^{-6}$  &   $3.1\times 10^{-4}$  &  $8.2 \times 10^{-9}$ & $2.1 \times 10^{-12}$   
  \end{tabular}
\end{table}\normalsize

There are a number of techniques that can often improve the performance of LSeig. For example, \emph{balancing} (diagonal similarity or diagonal equivalece transformation) is a standard technique for improving the conditioning of eigenvalue problems. As a simple version of this, we have observed that the residuals reported in the second and third columns of Table~\ref{tab:OS} improve by a factor of $10^{-2}$ to $10^{-3}$ if we divide $\quasi{A}$ by $\nu := \frac{\|\quasi{A}\|_F}{\|\quasi{B}\|_F} \approx 100$ and then recover $\lambda$'s by multiplying computed eigenvalues by $\nu$.

\end{example}

Our next example is an eigenvalue problem with eigenvalue-dependent boundary conditions. 
\begin{example}~\label{Chanane:ex} \normalfont
\emph{$\lambda$-dependent boundary conditions}.
\sloppy{We consider the problem $-u''(x) = \lambda u(x)$ where $x\in [0,1]$ with boundary conditions $-u(0) = (\lambda +d) u'(0)$ and $u(1) = \lambda u'(1)$ in which $d = -4\pi^2$. This is a problem taken from~\cite{binding97, Chanane05}. The first three eigenvalues of this problem have been computed in~\cite{Chanane05} by the regularized sampling method (RSM) which is based on the Shannon sampling theory. Using 100 Chebyshev polynomials and $\mbox{\texttt{tol}} = 10^{-9}$, our method computes 42 real eigenvalues in $0.14$ seconds. Table~\ref{Chanane:tab} reports our results with those available in~\cite{Chanane05} which sets ``precision'' to $10^{-10}$. In each instance, underlined digits are those which match the exact eigenvalues. Note that RSM is applicable even if boundary conditions are nonlinear in terms of $\lambda$.}
\begin{table}
 \caption{Three smallest real eigenvalues of the problem in Example~\ref{Chanane:ex}. \label{Chanane:tab}}
\begin{center}
\footnotesize 
        \begin{tabular}{c|c|c}
    exact & RSM \cite{Chanane05} \rule{0pt}{2.6ex} \rule[-1.2ex]{0pt}{0pt}& \texttt{LSeig} \rule{0pt}{2.6ex}\\
\hline
9.730886578213082033 & \underline{9.73088}7696302056807 \rule{0pt}{2.6ex} \rule[-1.2ex]{0pt}{0pt}& $\underline{9.7308865782}21018$ \\
\hline
88.76331625258976337 & \underline{88.763}23738197181406 \rule{0pt}{2.6ex} \rule[-1.2ex]{0pt}{0pt}& $\underline{88.76331625258}112$ \\
\hline
157.88411043863472059 & \underline{157.884}22274978466468 \rule{0pt}{2.6ex}\rule[-1.2ex]{0pt}{0pt}& $\underline{157.8841104386}164$ 
\end{tabular}
\normalsize
\end{center}
\end{table}

\end{example}

\section{Discussion}\label{sec:discuss}
Let us close with a high-level discussion of the different approaches of various methods. In the introduction we discussed the differences among spectral methods (collocation, Galerkin) and LSeig, which we summarize in Table~\ref{tab:summary}. Recall that LSeig has a flavor of both spectral collocation and Galerkin methods. 

\begin{table}[htbp]
\caption{Summary and comparison of spectral methods.}  \label{tab:summary}
\centering
\begin{tabular}{c|ccc}
& Galerkin &  collocation &  LS-eig \\\hline
$u=\sum c_i\phi_i$ represent  & $c_i$ & $u(x_i)$ & $c_i$\\
operator  & residual orth. & force at $x_i$ & least-squares \\
 boundary conditions& impose on $\phi_i$ & replace rows & least-squares 
\end{tabular}
\end{table}

Another intriguing class of methods is the method of fundamental solutions \cite{fox1967approximations}, revisited (among others) in \cite{betcke2005reviving}. These methods proceed in the `opposite' manner to spectral Galerkin/coefficient methods: whereas spectral Galerkin/coefficient methods use basis functions satisfying the boundary conditions and fit the operator, 
the method of fundamental solutions uses basis functions that satisfy the operator exactly, and then fits the boundary conditions, usually in the least-squares sense. In this sense LSeig can be regarded as a mixture of spectral coefficient methods and the method of fundamental solutions. We summarize the discussion in Table~\ref{tab:fund}. 

\begin{table}[htbp]
  \centering
\caption{Comparison between spectral methods, method of fundamental solutions and LSeig. }
\label{tab:fund}
  \begin{tabular}{c|ccc}
& Galerkin &  fund. soln &  LSeig \\\hline
opeartor & choose $c_i$ & $\checkmark$ & choose $c_i$\\
 boundary conditions & $\checkmark$ & choose $c_i$ & choose $c_i$ \\
  \end{tabular}
\end{table}

We believe the least-squares framework of LSeig and LSode would be an attractive alternative when nonstandard and tailor-made basis functions are available or can be computed. 

\section*{Acknowledgments}
The authors are grateful to Nick Trefethen for his encouragement, insightful discussions and valuable comments on a draft. They also thank Davoud Mirzaei for pointing out the connection to LSFEM. 

\bibliographystyle{siamplain}
\bibliography{references}

\ \\
{\bf Supplementary materials.} 
As explained before, a major difficulty with the generalized rectangular eigenvalue problem even in the discrete case is that the eigenpairs may fail to exist under perturbations. Motivated by the work \cite{boutry05} of Boutry, Elad, Golub, and Milanfar, we focus on the following reformulation of \cref{eigQuasi:eq} that searches for the minimal perturbation to the quasimatrix-matrix pencil $(\quasi{A}, \quasi{B})$ such that the perturbed pencil $(\quasi{\hat A}, \quasi{\hat B})$ has $n$ linearly independent eigenvectors: 
\begin{equation} 
\label{eq:eigQuasiMinPert}
\left\{ 
\begin{array}{ll}
\mbox{minimize}   &  \|[\quasi{\hat A} - \quasi{A}\ \ \ \quasi{\hat B} - \quasi{B}]\|_F^2;\\
\mbox{subject to} & \quasi{\hat A}, \quasi{\hat B}\in\mathbb{C}^{(\infty+d)\times n},\ \{(\lambda_k, {\bf v}_k)\}_{k=1}^{n} \subseteq \mathbb{C} \times \mathbb{C}^n,\\
                            & \quasi{\hat A} {\bf v}_k = \lambda_k \quasi{\hat B} {\bf v}_k, \quad k = 1,2,\dots,n, \\
                            & \{{\bf v}_1, {\bf v}_2, \dots, {\bf v}_n \}:\ \mbox{linearly independent}.
\end{array} 
\right.
\end{equation}

The next proposition gives a sufficient condition for the existence and uniqueness of the optimal solution to \eqnref{eq:eigQuasiMinPert} and hints an algorithm for its solution in terms of the following SVD\footnote{See \eqnref{eq:svdQuasiAB} for details but notice that it was for the case of $[\quasi{A} \ \quasi{B}] $.}.
\begin{equation} 
\label{eq:svdQuasiBA}
[\quasi{B}\ \ \quasi{A}]=\quasi{U}\Sigma V^* = 
[\quasi{U}_1\ \quasi{U}_2]
\begin{bmatrix}
\Sigma_{1:n}&\\& \Sigma_{n+1:2n}  
\end{bmatrix}
\begin{bmatrix}
V_{11}^*& V_{21}^*\\
V_{12}^*& V_{22}^*
\end{bmatrix}, 
\end{equation}
It is a direct extension of the results from the discrete case~\cite{golub80TLS, itomurota2016} to quasimatrix-matrix objects; see also \cite[p. 51]{vanhuffel91book}. To keep the paper self-contained we give a proof of the proposition which requires a few definitions and the following three lemmas. For a quasimatrix-matrix $\quasi{A} \in \mathbb{C}^{(\infty + d) \times n}$ we define
\begin{equation}
\label{2normDef:eq}
\|\quasi{A}\|_2 := \max_{0 \neq {\bf x} \in \mathbb{C}^n} \frac{\|\quasi{A} {\bf x}\|_2}{\| {\bf x}\|_2},
\end{equation}
and 
\[
\|\quasi{A}\|_F := \Big( \mbox{trace}(\quasi{A}^{\ast} \quasi{A})\Big)^{1/2} = \Big( \sum_{i=1}^n \sigma_{i}^2(\quasi{A}) \Big)^{1/2}.
\]
In addition, $\quasi{A}$ is called unitary if $\quasi{A}^\ast \quasi{A} = I_{n}$.

\begin{lemma} \label{2norm_uni_inv:lem}
\normalfont
For any $W \in \mathbb{C}^{n \times k}$ and unitary quasimatrix-matrix $\quasi{U} \in \mathbb{C}^{(\infty + d) \times n}$, we have
\[
\|\quasi{U} W\|_2 =\|W\|_2.
\]
\end{lemma}
\begin{proof}
For every ${\bf x}\in \mathbb{C}^n$ we have $\|\quasi{U} {\bf x}\|_2 = \|(\quasi{U} {\bf x})^\ast (\quasi{U} {\bf x})\|_2 = \| {\bf x}^\ast I_n {\bf x}\|_2 = \|{\bf x}\|_2$. The invariance of the 2-norm follows from~\eqnref{2normDef:eq}.
\end{proof}

The following continuous analogue of the well-known Eckart-Young-Mirsky theorem will be used. It is stated in terms of the Frobenius norm of a quasimatrix-matrix. 
\begin{lemma} \label{bestlowrankQuasi:lem} 
The first k-terms in the SVD of a quasimatrix-matrix $\quasi{A} \in \mathbb{C}^{(\infty+d) \times n}$ form its best rank-k approximation in the Frobenius norm.
\end{lemma}
\begin{proof}
The reasoning is analogous to that of quasimatrices as in~\cite[p. 62]{townsend14} and~\cite{townsend15}. 
\end{proof}

The next lemma is a continuous analogue of the the result in \cite[p. 321]{GVLbook13} for the discrete case.
\begin{lemma} \label{TLSsing:lem}
\normalfont
Let $\quasi{A}, \quasi{B} \in  \mathbb{C}^{(\infty+d) \times n}$ and consider the SVD~\eqnref{eq:svdQuasiBA} of $[\quasi{B}\ \ \quasi{A}]$. If $\sigma_{n}(\quasi{B})> \sigma_{n+1}([\quasi{B}\ \ \quasi{A}])$, then $V_{11}$ and $V_{22}$ are nonsingular and $\sigma_{n}([\quasi{B}\ \ \quasi{A}])> \sigma_{n+1}([\quasi{B}\ \ \quasi{A}])$. 
\end{lemma}
\begin{proof} 
We first use proof by contradiction to show that $V_{22}$ is nonsingular. Assume that there exists a vector ${\bf x}$ with unit 2-norm such that $V_{22} {\bf x} = 0$.
The second equation in $[\quasi{B}\ \ \quasi{A}] V = [\quasi{U}_1\ \ \quasi{U}_2] \Sigma$ reads as $\quasi{B} V_{12} + \quasi{A} V_{22} = \quasi{U}_2 \Sigma_{n+1:2n}$ meaning that 
\begin{equation}
\label{formula1:eq}
\| \quasi{B} V_{12} {\bf x} \|_2 = \| \quasi{U}_2 \Sigma_{n+1:2n} {\bf x}\|_2.
\end{equation}
As we saw in Lemma~\ref{2norm_uni_inv:lem}, the 2-norm is invariant under multiplication by a quasimatrix-matrix like $\quasi{U}_2$ whose columns are orthonormal function-vectors. Therefore, using \eqnref{formula1:eq} we have
\[
\sigma_{n+1}([\quasi{B}\ \ \quasi{A}]) = \| \Sigma_{n+1:2n} \|_2 = \| \quasi{U}_2 \Sigma_{n+1:2n} \|_2 \geq \| \quasi{U}_2 \Sigma_{n+1:2n}  {\bf x} \|_2 = \| \quasi{B} V_{12} {\bf x} \|_2 \geq \sigma_{\min}(\quasi{B}),
\]
which is a contradiction. 

The second part follows if we prove that $\sigma_{n}([\quasi{B}\ \ \quasi{A}]) \geq \sigma_{n}(\quasi{B})$. This is an interlacing property for singular values of a quasimatrix-matrix which is valid because the singular values of any quasimatrix-matrix are just the singular values of the $R$ factor of its QR factorization (see~\eqnref{svd_R_concat:eq} and~\eqnref{eq:svdRfactorquasimat-mat}) and the $R$ factor is always a discrete matrix for which the interlacing property of singular values is a basic fact~\cite{thompson72}.
\end{proof}

\begin{proposition}\label{optSol:prop}
Let $\quasi{A}, \quasi{B} \in  \mathbb{C}^{(\infty+d) \times n}$ and consider the SVD~\eqnref{eq:svdQuasiBA} of $[\quasi{B}\ \ \quasi{A}]$. If $\sigma_{n}(\quasi{B})> \sigma_{n+1}([\quasi{B}\ \ \quasi{A}])$, then there exists a unique optimal solution to 
\eqnref{eq:eigQuasiMinPert} 
attained for
\[
\quasi{\hat A} = \quasi{A} - \quasi{U}_2 \Sigma_{n+1:2n} V_{22}^\ast, \quad \mbox{ and }  \quad
\quasi{\hat B} = \quasi{B} - \quasi{U}_2 \Sigma_{n+1:2n} V_{12}^\ast,
\]
if and only if $V_{12} V_{22}^{-1}$ is diagonalizable.
\end{proposition}
\begin{proof}
We prove the result in two steps. In the first step we just extend the argument by Ito and Murota~\cite[Thm. 2, part i)]{itomurota2016} to the case of quasimatrix-matrix objects by showing that \eqnref{eq:eigQuasiMinPert} is equivalent to the following continuous-discrete total least-squares problem
\begin{equation} 
\label{eq:eigQuasiTLS}
\left\{ 
\begin{array}{ll}
\mbox{minimize}   &  \|[\quasi{\hat A} - \quasi{A}\ \ \ \quasi{\hat B} - \quasi{B}]\|_F^2;\\
\mbox{subject to} & \quasi{\hat A}, \quasi{\hat B}\in\mathbb{C}^{(\infty+d)\times n}\\
                            & \mbox{ range} (\quasi{\hat A}) \subseteq \mbox{range} (\quasi{\hat B}).
\end{array} 
\right.
\end{equation}
Let $P_1$ denote the set of all feasible solutions to \eqnref{eq:eigQuasiMinPert} and 
assume that $\quasi{\hat A}, \quasi{\hat B}$ and $\{(\lambda_k, {\bf v}_k)\}_{k=1}^{n}$ are one of those feasible solutions. Assuming $V := [{\bf v}_1,\ {\bf v}_2,\ \dots, {\bf v}_n]$ and $\Lambda := \diag(\lambda_1, \lambda_2, \dots, \lambda_n)$, the constraints $\quasi{\hat A} {\bf v}_k = \lambda_k \quasi{\hat B} {\bf v}_k$ for $k = 1,2,\dots,n$ means that $\quasi{\hat A} V = \quasi{\hat B} V \Lambda$ and since columns of $V$ are linearly independent, \eqnref{eq:eigQuasiMinPert} is equivalent to 
$\quasi{\hat A} = \quasi{\hat B} V \Lambda V^{-1}$. This representation shows that $P_1$ is the same as the set of all $(\quasi{\tilde A}, \quasi{\tilde B}) \in \mathbb{C}^{(\infty+d) \times n} \times \mathbb{C}^{(\infty+d) \times n}$ satisfying $\quasi{\tilde A} = \quasi{\tilde B} Z$ where $Z \in \mathbb{C}^{n \times n}$ is diagonalizable.

Let $P_2$ denote the set of all $(\quasi{\tilde A}, \quasi{\tilde B}) \in \mathbb{C}^{(\infty+d) \times n} \times \mathbb{C}^{(\infty+d) \times n}$ satisfying $\quasi{\tilde A} = \quasi{\tilde B} Z$ where $Z \in \mathbb{C}^{n \times n}$ is {\em not} necessarily diagonalizable.\ This means that $P_2$ is the set of all $(\quasi{\tilde A}, \quasi{\tilde B})$ such that $\mbox{range} (\quasi{\tilde A}) \subseteq \mbox{range} (\quasi{\tilde B})$.\ Obviously $P_1 \subseteq P_2$. Since there exists a diagonalizable matrix in an arbitrarily close neighborhood of any square matrix $Z$, we have $P_2 \subseteq \overline{P_1}$ where $\overline{P_1}$ denotes the closure of $P_1$. In addition, $\|[\quasi{\tilde A} - \quasi{A}\ \ \ \quasi{\tilde B} - \quasi{B}]\|_F^2$ is a continuous function of $(\quasi{\tilde A}, \quasi{\tilde B})$. Therefore,
\[
\inf_{(\quasi{\tilde A}, \quasi{\tilde B}) \in P_1} \{ \|[\quasi{\tilde A} - \quasi{A}\ \ \ \quasi{\tilde B} - \quasi{B}]\|_F^2 \} = \inf_{(\quasi{\tilde A}, \quasi{\tilde B}) \in P_2} \{ \|[\quasi{\tilde A} - \quasi{A}\ \ \ \quasi{\tilde B} - \quasi{B}]\|_F^2 \},
\]
which means that optimal solutions to \eqnref{eq:eigQuasiMinPert} and \eqnref{eq:eigQuasiTLS} are the same.

Now in the second step we first derive explicit formulas for the unique optimal solution to \eqnref{eq:eigQuasiTLS} (and according to the first step an optimal solution to \eqnref{eq:eigQuasiMinPert} as well.) The assumption $\sigma_{n}(\quasi{B})> \sigma_{n+1}([\quasi{B}, \quasi{A}])$ together with Lemma~\ref{TLSsing:lem} implies that $V_{11}$ is nonsingular and that $\sigma_{n}([\quasi{B}\ \ \quasi{A}])> \sigma_{n+1}([\quasi{B}, \quasi{A}])$. We rewrite the formula $\quasi{\hat A} = \quasi{\hat B} Z$ as 
\[
[\quasi{\hat B}\ \  \quasi{\hat A}] \begin{bmatrix}
Z\\
-I
\end{bmatrix} = 0,
\]
implying that the rank of the augmented quasimatrix-matrix $[\quasi{\hat B}\ \ \quasi{\hat A}]$ is at most $n$. Therefore, we can view solving \eqnref{eq:eigQuasiTLS} as finding the minimal (in the Frobenius norm) rank-n perturbation $[\quasi{\hat B}\ \ \quasi{\hat A}]$ to $[\quasi{B}\ \ \quasi{A}]$. According to Lemma ~\ref{bestlowrankQuasi:lem}, the latter problem can be solved by the rank-n truncation of the SVD \eqnref{eq:svdQuasiBA}, i.e., 
\[
[\quasi{\hat B}\ \  \quasi{\hat A}] = \quasi{U}_1 \Sigma_{1:n} [V_{11}^\ast\ \ V_{21}^\ast],
\]
which is unique as $\sigma_{n}([\quasi{B}\ \ \quasi{A}])> \sigma_{n+1}([\quasi{B}\ \ \quasi{A}])$. To find the corresponding solution $\hat Z$ to $\quasi{\hat B} Z = \quasi{\hat A}$ we therefore put
\[
\quasi{U}_1 \Sigma_{1:n} V_{11}^\ast Z = \quasi{U}_1 \Sigma_{1:n}  V_{21}^\ast. 
\]
Since $V_{11}$ is nonsingular, $\hat Z = (V_{21} V_{11}^{-1})^\ast$ solves $\quasi{\hat B} Z = \quasi{\hat A}$ in~\eqnref{eq:eigQuasiTLS}. From the orthogonality of the partitioned matrix $V$ it follows that $(V_{21} V_{11}^{-1})^\ast = - V_{12} V_{22}^{-1}$.

On the other hand if $\hat Z$ corresponding with the optimal solution $[\quasi{\hat B}\ \ \quasi{\hat A}]$ to~\eqnref{eq:eigQuasiTLS} is diagonalizable, then its eigenpairs $\{(\lambda_k, v_k)\}_{k=1}^{n}$ satisfy a representation of the form $\hat Z = V \Lambda V^{-1}$ which by $\quasi{\hat B} \hat Z = \quasi{\hat A}$ means that $(\quasi{\hat B}, \quasi{\hat A})$ is an optimal solution to~\eqnref{eq:eigQuasiMinPert}. Conversely, if 
$(\quasi{\hat B}, \quasi{\hat A})$ is an optimal solution to~\eqnref{eq:eigQuasiMinPert}, then $\quasi{\hat A} = \quasi{\hat B} V \Lambda V^{-1}$ which means that $\hat Z = V \Lambda V^{-1}$ is diagonalizable.
\end{proof}

\end{document}